\documentclass[a4paper, 11pt, twoside]{article}
\pdfoutput=1
\usepackage{lspaper,xcolor,scalefnt,tikz,pgfkeys}
\usepackage{blkarray,mathdots}
\usepackage{quiver}
\usepackage{tcolorbox}

\newgeometry{margin=0.8in}

\DeclareMathOperator{\het}{ht}

\DeclareTextCommand{\textgamma}{PU}{\83\263}

\newcommand\uD[1]{\underline{\operatorname{D}}_{#1}}
\newcommand\uspe[1]{\underline{\operatorname{S}}_{#1}}

\newcommand\sL[1]{\operatorname{L}\left(#1\right)}
\newcommand\sM[1]{\operatorname{M}\left(#1\right)}
\newcommand\sY[1]{\operatorname{Y}\left(#1\right)}

\newcommand\sLx[1]{\operatorname{L}_{#1}}

\newcommand{\clapdots}{\; \mathclap{\dots} \;}

\renewcommand{\phi}{\varphi}

\newcommand{\s}[2]{s\hspace{-4pt}\underset{\scriptscriptstyle{#2}}{\overset{\scriptscriptstyle{#1}}{\downarrow}}}
\newcommand{\su}[2]{s\hspace{-4pt}\overset{\scriptscriptstyle{#2}}{\underset{\scriptscriptstyle{#1}}{\uparrow}}}

\newcommand\wh{\Yfillcolour{white}}

\definecolor{bittersweet}{rgb}{1.0, 0.44, 0.37}
	\definecolor{lightcoral}{rgb}{0.94, 0.5, 0.5}

\newcommand\lc{\Yfillcolour{bittersweet}}

\def\bgg{\text{\boldmath$g$}}

\def\b1{\text{\boldmath$1$}}

\begin{document}

\title{Decomposable Specht modules indexed by bihooks II}

\author{Robert Muth\\\normalsize Washington \& Jefferson College\\\normalsize Washington, PA, USA 15301 \\\texttt{\normalsize rmuth@washjeff.edu}\\[11pt] Liron Speyer
\\\normalsize Okinawa Institute of Science and Technology\\\normalsize Okinawa, Japan 904-0495 \\\texttt{\normalsize liron.speyer@oist.jp}\\[11pt] Louise Sutton\footnote{Previous address: University of Manchester and Heilbronn Institute for Mathematical Research, UK M13 9PL}
\\\normalsize Okinawa Institute of Science and Technology\\\normalsize Okinawa, Japan 904-0495\\
\texttt{\normalsize louise.sutton@oist.jp}
}

\renewcommand\auth{Robert Muth, Liron Speyer \& Louise Sutton}

\runninghead{Decomposable Specht modules indexed by bihooks II}

\toptitle

\begin{abstract}
Previously, the last two authors found large families of decomposable Specht modules labelled by bihooks, over the Iwahori--Hecke algebra of type $B$.
In most cases we conjectured that these were the only decomposable Specht modules labelled by bihooks, proving it in some instances.
Inspired by a recent semisimplicity result of Bowman, Bessenrodt and the third author, we look back at our decomposable Specht modules and show that they are often either semisimple, or very close to being so. 
We obtain their exact structure and composition factors in these cases. 
In the process, we determine the graded decomposition numbers for almost all of the decomposable Specht modules indexed by bihooks.
\end{abstract}


\section{Introduction}

The Specht modules are a natural family of modules for the symmetric group, indexed by partitions, whose combinatorial construction has natural generalisations to Hecke algebras, and more generally to cyclotomic Hecke algebras, where \emph{multipartitions} play the corresponding indexing role.
In all of these cases, these modules are the ordinary irreducible modules for these algebras, and it is an open and difficult problem to determine their structure outside of the semisimple situation.
One natural question one can ask is when these important modules are \emph{decomposable}.
For the symmetric groups, an old result of James \cite[Corollary 13.18]{James} tells us that the Specht modules are always indecomposable over fields of characteristic not equal to $2$; this has an analogue for their Hecke algebras too -- the Specht modules are indecomposable outside of \emph{quantum characteristic} $e=2$.
Already in this setting, it is an unresolved, hard problem to determine which Specht modules are decomposable when $e=2$.
Some progress has been made by Murphy \cite{gm80} and the second author \cite{ls14}, where decomposable Specht modules indexed by \emph{hook partitions} were classified in the symmetric group and Hecke algebra cases, respectively.
Some more decomposable Specht modules for symmetric groups and their Hecke algebras can be found in \cite{df12,dg20,bbs19}.

The analogous question for cyclotomic Hecke algebras has not yet been studied, until the prequel to this paper~\cite{ss20}.
There, the second and third authors initiated this study for (integral) Hecke algebras of type $B$, which may be seen as level 2 cyclotomic Hecke algebras.
These algebras have Specht modules $\spe\la$ that are now indexed by \emph{bipartitions} $\la$, and it is known (for example by~\cite[Corollary 3.12]{fs16}) that all Specht modules are indecomposable unless $e=2$ or the \emph{bicharge} is of the form $\kappa = (i,i)$ for some $i$.
(Up to isomorphism we may assume that $i=0$ in this case.)
For $\kappa =(0,0)$, we found a large family of decomposable Specht modules indexed by bihooks, that is bipartitions where each component consists of a hook, a natural generalisation of the hook partitions previously studied.
The main results are summarised in the below theorem, where $p$ denotes the characteristic of the ground field $\bbf$.

\begin{thmc}{ss20}{Theorems 4.1 and 5.4}\label{thm:decompbihooks}
Let $\kappa = (0,0)$, and $\la = ((ke + a, 1^b),(je + a, 1^b))$ or $((b+1, 1^{je+a-1}), (b+1, 1^{ke+a-1}))$, for some $j,k\geq 1$, $0< a \leq e$ and $0 \leq b < e$ with $a+b \neq e$, or for $a=b=0$.
	\begin{enumerate}[label=(\roman*)]
		\item For $j,  k > 1$, if $j+k$ is even and $p \neq 2$, or if $j+k$ is odd, then $\spe\la$ is decomposable.
		
		\item If $j=1$ or $k=1$, then $\spe\la$ is decomposable if and only if $p \nmid j+k$.
	\end{enumerate}
\end{thmc}

%
%

As noted in the proof of~\cite[Theorem~5.4]{ss20}, if $e=2$, then $\spe{((ke),(je))} \cong \spe{((ke),(1^{je}))}$, since the presentations of these two Specht modules are identical.
This allows the above results to also cover $\la = ((2k+a, 1^b),(a, 1^{2j+b}))$ or $((a,1^{2k+b}),(2j+a, 1^b))$.
We conjectured~\cite[Conjecture~4.2]{ss20} that if $e\neq 2$ and $p \neq 2$, these are all of the decomposable Specht modules indexed by bihooks.

As in the previous paper, our approach here uses the cyclotomic KLR algebras introduced by Khovanov, Lauda and Rouquier (\cite{kl09,rouq}), via the now-famous isomorphism of  Brundan and Kleshchev~\cite{bkisom}.
This perspective allows for the study of the \emph{graded} representation theory of cyclotomic Hecke algebras.

We observed -- using the LLT algorithm -- that, over a field of characteristic $0$, the above Specht modules seemed to have all composition factors focussed in a single degree.
A recent result of Bowman, Bessenrodt and the third author~\cite[Theorem~3.2]{bbs19} then implies that the Specht modules must be semisimple.
The purpose of the current paper is to study the structure of these decomposable Specht modules we found previously, and in particular determine when they are semisimple.

Our first main result determines precisely when the Specht modules are semisimple, and determines the summands (the composition factors, with grading shift).
It appears in the body of the paper as \cref{thm:k>jss}.

\begin{thm}
Suppose $k \geq j\geq 1$.
Then $\spe{((ke),(je))}$ is semisimple if and only if one of the following holds:
\begin{itemize}
\item $p\neq 2$ and $p$ does not divide any of the integers $k+j, k+j-1, \dots, k-j+2$;

\item $p=2$, $j=1$, and $k$ is even;

\item $p=2$, $j=2$, and $k\equiv 1 \pmod 4$.
\end{itemize}
When semisimple, $\spe{((ke),(je))}$ is isomorphic to
\[
\bigoplus_{r=1}^{j} \D{\left( \left((k+j-r)e,(r-1)e+1\right),(e-1) \right)}\langle j\rangle \oplus \D{((ke+je),\varnothing)}\langle j\rangle.
\]
\end{thm}

From this, \cref{thm:generalised k>j ss,cor:generalised k>j ss transpose} give the corresponding result for the more general bihooks of \cref{thm:decompbihooks}, by arguments using graded $i$-induction functors, which were also key in the proof of \cref{thm:decompbihooks}.

Our methods for this use a Morita equivalence of Kleshchev and the first author~\cite{km17}, that allows us to pass the Specht modules $\spe{((ke),(je))}$ to the tensor product of Weyl modules $\Delta(1^k) \otimes \Delta(1^j)$ over the usual Schur algebra $S(k+j,k+j)$ -- see \cref{lem:moritaequiv}.
We are then able to perform a large part of our analysis in the Schur algebra setting, pulling results back through the Morita equivalence.
This analysis on the Schur algebra side is supplemented on the KLR algebra side by some structure that we are able to extract from Specht module homomorphisms that we define along the way.

Recasting our problem in the Schur algebra setting allows us to also gain traction in some of the non-semisimple cases.
In particular we are able to give the complete structure of the Specht modules appearing in \cref{thm:decompbihooks} whenever $j=1$ or $2$, or $p$ divides precisely one of the integers $k+j, k+j-1, \dots, k-j+2$ (see \cref{cor:j=1induced,cor:j=1induced transpose},  \cref{prop:j=2,cor:j=2induced}, and \cref{thm:pdiv1}, respectively).
In most of these cases, our analysis hinges on piecing together Weyl modules that are either simple or have only two composition factors, which simplifies the situation. 
In future work, we hope to consider such situations more generally, though the answer is less clean than our current results, for instance see \cref{eg:5factorsummand}.
Despite not determining their structure in complete generality, along the way we obtain the graded decomposition numbers for all Specht modules indexed by bihooks of the form found in \cref{thm:decompbihooks} -- see the remark after \cref{lem:charchange}.

In \cref{cor:p=2decomposables}, we determine precisely when our Specht modules are decomposable for $j,k>1$ and $p=2$, filling a gap in \cref{thm:decompbihooks}.

In ongoing work, we are utilising the methods and results from this paper to yield presentations and graded dimension formulae for some simple modules in level 2.
In fact, we are studying the simple modules appearing in \cref{thm:k>jss,thm:generalised k>j ss}, under the same conditions.

Our paper is organised as follows.
In \cref{sec:background}, we lay the foundations for our work, recalling the necessary Lie theory setup, tableau-combinatorics, as well as definitions and some properties of KLR algebras and their Specht modules.
In \cref{sec:imag}, we recall work of Kleshchev and the first author \cite{km17a,km17} studying imaginary Schur algebras, and apply their results to our situation.
The main outcome of this section is a Morita equivalence that allows us to largely transport our problem to the classical Schur algebra, where we only need to work with Weyl modules indexed by two-column partitions.
\cref{sec:homs} is devoted to proving \cref{prop:alphahom,prop:gammahom}, which provide homomorphisms between certain Specht modules.
These will provide the driving force for our proofs from one side of the Morita equivalence of \cref{sec:imag}.
Our main results come in \cref{sec:main}, where we use our Morita equivalence to combine information from our Specht module homomorphisms with some known results for Weyl modules (over the Schur algebra) indexed by two-column partitions and obtain our main theorem on semisimple Specht modules, \cref{thm:k>jss}, as well as determine the structures of some non-semisimple Specht modules.

\begin{ack}
The authors thank Kay Jin Lim for useful comments about tilting modules, 
and Matthew Fayers, whose GAP packages were used in LLT computations as well as homomorphism computations.
The authors are also grateful for the support from Singapore MOE Tier 2 AcRF MOE2015-T2-2-003, which funded a research visit of the first two authors to the third at the National University of Singapore.
The second author is partially supported by JSPS Kakenhi grant number 20K22316.
Finally, we thank the referees for carefully reading the paper and suggesting numerous improvements.
\end{ack}

\section{Background}\label{sec:background}

In this section we give an overview of KLR algebras, Specht modules labelled by bihooks, and the associated combinatorics, as in our previous paper.
For our methods here, we will also need to recall some information about Weyl modules over Schur algebras, and their connection to KLR algebras via Imaginary Schur--Weyl duality.
Throughout, $\bbf$ will denote an arbitrary field of characteristic $p\geq 0$.

\subsection{Lie theoretic notation}\label{subsec:Lie}

Let $e\in \{2,3,\dots\} $, which we call the \emph{quantum characteristic}.
We set $I:=\bbz/{e\bbz}$, which we identify with the set $\{0, 1, \dots, e-1\}$.
We let $\Gamma$ be the (type $A_{e-1}^{(1)}$) quiver with vertex set $I$ and an arrow $i\rightarrow{i-1}$ for each $i\in{I}$.

Following Kac's book~\cite{kac}, we recall standard notation for the Kac--Moody algebra associated to the generalised Cartan matrix $(a_{ij})_{i,j\in I}$.
Explicitly, $a_{ij} = 2\delta_{i,j} - \delta_{i,j+1} - \delta_{i,j-1}$.
We have simple roots $\{\alpha_i\mid{i\in{I}}\}$, fundamental dominant weights $\{\La_i \mid i \in I\}$, and the invariant symmetric bilinear form $(\:,\:)$ such that $(\alpha_i, \alpha_j) = a_{i,j}$ and $(\La_i, \alpha_j) = \delta_{i,j}$, for all $i, j \in I$. Let \(\Phi_+\) be the set of positive roots, and let $Q_+:= \bigoplus_{i\in I} \bbz_{\geq 0} \alpha_i$ be the positive cone of the root lattice.
We write \(\delta:= \alpha_0 + \cdots + \alpha_{e-1} \in \Phi_+\) for the {\em null root}.
If $\alpha = \sum_{i\in I} c_i \alpha_i \in Q_+$, then we define the \emph{height of $\alpha$} to be $\het(\alpha) = \sum_{i \in I} c_i$. 

An \emph{$e$-bicharge} is an ordered pair $\kappa = (\kappa_1,\kappa_2)\in{I^2}$.
We define its associated dominant weight $\La$ of level two to be $\La = \La_\kappa := \La_{\kappa_1} + \La_{\kappa_2}$.

\subsection{The symmetric group}

Let $\mathfrak{S}_n$ be the symmetric group on $n$ letters.
We let $s_1,\dots,s_{n-1}$ denote the standard Coxeter generators, where $s_i$ is the simple transposition $(i, i+1)$ for $1\leq i<n$.
We define a \emph{reduced expression} for a permutation $w\in\mathfrak{S}_n$ to be an expression $s_{i_1}\dots s_{i_m}$ such that $m$ is minimal, and call $m$ the length of $w$, denoted $\ell(w)$.

For \(h,n \in \mathbb{N}\), we will let \(S_\bbf(h,n)\) denote the classical Schur algebra over \(\bbf\); see for instance \cite[\S1]{bdk01} for definitions and representation theoretic details.

We define the \emph{Bruhat order} $\leq$ on $\mathfrak{S}_n$ as follows.
If $x, w \in \mathfrak{S}_n$, then we write $x\leq w$ if there is a reduced expression for $x$ which is a subexpression of a reduced expression for $w$.

\subsection{Bipartitions}

A \emph{composition} $\la$ of $n$ is a  sequence of non-negative integers $\la = (\la_1, \la_2, \dots)$ such that $|\la| := \sum \la_i = n$.
We say \(\la\) is a \emph{partition} of $n$ provided that the sequence is weakly decreasing, i.e. \(\la_i \geq \la_j\) for all \(i \leq j\). We write $\varnothing$ for the \emph{empty partition} $(0, 0, \dots)$.
We will denote the set of all compositions of \(n\) by \(\mathscr{C}_n\), and the set of all partitions of \(n\) by \(\mathscr{P}_n\). We also define:
\begin{align*}
\mathscr{C}_n(a) &= \{ \la \textup{ a composition of \(n\)} \mid \la_i = 0 \textup{ for all }i>a\}\\
\mathscr{P}_n(a) &= \{ \la \textup{ a partition of \(n\)} \mid\la_i = 0 \textup{ for all }i>a\}
\end{align*}
for the sets of compositions and partitions of \(n\), respectively, with {\em \(a\) or fewer parts}.

A \emph{bipartition} $\la$ of $n$ is a pair $\la = (\la^{(1)}, \la^{(2)})$ of partitions such that $|\la| = |\la^{(1)}| + |\la^{(2)}| = n$.
We refer to $\la^{(1)}$ and $\la^{(2)}$ as the \emph{$1$st  and $2nd$ component}, respectively, of $\la$.
We abuse notation and also write $\varnothing$ for the \emph{empty bipartition} $(\varnothing, \varnothing)$.
We denote the set of all bipartitions of $n$ by $\mptn 2 n$.

\begin{defn}
We call a bipartition $\la$ a \emph{bihook} if it is of the form $\la = ((a,1^b),(c,1^d))$ for some integers $a,c \geq1$ and $b,d \geq 0$.
\end{defn}

For $\la, \mu \in \mptn 2 n$, we say that $\la$ \emph{dominates} $\mu$, and write $\la \dom \mu$, if for all $k\geq 1$,
\[
\sum_{j=1}^{k} \la_j^{(1)} \geq \sum_{j=1}^{k} \mu_j^{(1)} \text{ and } |\la^{(1)}| + \sum_{j=1}^{k} \la_j^{(2)} \geq |\mu^{(1)}| + \sum_{j=1}^{k} \mu_j^{(2)}.
\]
The \emph{Young diagram} of $\la = (\la^{(1)}, \la^{(2)}) \in \mptn 2 n$ is defined to be
\[
[\la]:= \{ (i,j,m)\in \bbn \times \bbn \times \{1, 2\} \mid 1\leq j \leq \la_i^{(m)}\}.
\]
We refer to elements of $[\la]$ as \emph{nodes} of $\la$.
We draw the Young diagram of a bipartition as a column vector of Young diagrams $[\la^{(1)}], [\la^{(2)}]$ from top to bottom.
We say that a node $A\in[\la]$ is \emph{removable} if $[\la]\setminus \{A\}$ is a Young diagram of a bipartition, while a node $A\not\in[\la]$ is \emph{addable} if $[\la]\cup\{A\}$ is a Young diagram of a bipartition.
We say that the node $(r,c,m)$ is \emph{above} the node $(r',c',m')$ if either $m < m'$ or ($m = m'$ and
$r < r'$), which is equivalent to saying that $(r',c',m')$ is \emph{below} the node $(r,c,m)$.

If $\la$ is a partition, the \emph{conjugate partition}, denoted $\la'$, is defined by
\[
\la_i'=|\left\{j\geq 1\mid\la_j\geq i\right\}|.
\]
If $\la \in \mptn 2 n$, then we define the conjugate bipartition, also denoted $\la'$, to be $\la' = (\la^{(2)'},\la^{(1)'})$.

\subsection{Tableaux}

Let $\la \in \mptn 2 n$.
Then a \emph{$\la$-tableau} is a bijection $\ttt:[\la] \rightarrow \{1, \dots, n\}$.
We depict a $\la$-tableau $\ttt$ by inserting entries $1, \dots, n$ into the Young diagram $[\la]$ with no repeats; we let $\ttt(i,j,m)$ denote the entry lying in node $(i,j,m) \in [\la]$.
We say that $\ttt$ is \emph{standard} if its entries increase down each column and along each row, within each component, and denote the set of all standard $\la$-tableaux by $\std\la$.

The \emph{column-initial tableau} $\ttt_\la$ is the $\la$-tableau where the entries $1,\dots,n$ appear in order down consecutive columns, working from left to right, first in component $2$, then in component $1$.

The symmetric group $\mathfrak{S}_n$ acts naturally on the left on the set of all $\la$-tableaux.
For $\ttt$ a $\la$-tableau, we define the permutation $w_\ttt \in \mathfrak{S}_n$ by $w_\ttt \ttt_\la =\ttt$.

Suppose $\la \in \mptn 2 n$.
Let $\tts$ and $\ttt$ be $\la$-tableaux with corresponding reduced expressions $w_{\tts}$ and $w_{\ttt}$, respectively.
Then we say that \emph{$\ttt$ dominates $\tts$}, written as $\ttt \dom \tts$, if and only if $w_{\ttt}\geq w_{\tts}$.

Fix an $e$-bicharge $\kappa = (\kappa_1,\kappa_2)$.
The \emph{$e$-residue} of a node $A = (i,j,m) \in \bbn\times\bbn\times\{1,2\}$ is defined to be
\[
\res A := \kappa_m+j-i \pmod{e}.
\]
We call a node of residue $r$ an $r$-\emph{node}.

Let $\ttt$ be a $\la$-tableau.
If $\ttt(i,j,m) = r$, we set $\res_\ttt (r)=\res (i,j,m)$.
The \emph{residue sequence} of $\ttt$ is defined to be
\[
\bfi_\ttt = (\res_\ttt (1),\dots,\res_\ttt (n)).
\]
We denote the residue sequence of the column-initial tableau $\ttt_\la$ by $\bfi_\la:=\bfi_{\ttt_\la}$.

We now define the degree and codegree of a standard tableau, as in \cite[\S3.5]{bkw11}.
Note that in \cite{ss20} we did not need the degree, and we thus used the word degree to refer to the codegree.
For $\la \in \mptn 2 n$ and an $i$-node $A$ of $\la$, we define
\begin{align*}
d_A(\la):&=
\#\left\{
\text{addable $i$-nodes of $\la$ below $A$}
\right\}
-\#\left\{
\text{removable $i$-nodes of $\la$ below $A$}
\right\};\\
d^A(\la):&=
\#\left\{
\text{addable $i$-nodes of $\la$ above $A$}
\right\}
-\#\left\{
\text{removable $i$-nodes of $\la$ above $A$}
\right\}.
\end{align*}

Let $\ttt \in \std\la$ with $\ttt^{-1}(n) = A$.
We define the \emph{degree} and \emph{codegree} of $\ttt$, denoted $\deg(\ttt)$ and $\codeg(\ttt)$, recursively, by setting $\deg(\varnothing):=0=:\codeg(\varnothing)$, and  
\[
\deg (\ttt):=d_A(\la)+\deg (\ttt_{<n}), \quad \codeg (\ttt):=d^A(\la)+\codeg (\ttt_{<n}),
\]
where $\ttt_{<n}$ is the standard tableau obtained from $\ttt$ by removing the node $A$.

\subsection{KLR algebras and their cyclotomic quotients}

Suppose $\alpha\in Q^+$ has height $n$, and set
\[
I^\alpha = \lset{\bfi = (i_1, i_2, \dots, i_n)\in I^n}{\alpha_{i_1}+\dots+\alpha_{i_n} = \alpha}.
\]
We define the \emph{Khovanov--Lauda--Rouquier (KLR) algebra} or \emph{quiver Hecke algebra} $\mathscr{R}_\alpha$ to be the unital associative $\bbf$-algebra with generating set
\[
\lset{e(\bfi)}{\bfi \in I^\alpha}\cup\{y_1,\dots,y_n\}\cup\{\psi_1,\dots,\psi_{n-1}\}
\]
and relations
{\allowdisplaybreaks
\begin{align*}
e(\bfi)e(\bfj) &= \delta_{\bfi,\bfj} e(\bfi);\\
\sum_{i \in I^\alpha} e(\bfi) &= 1;\\
y_re(\bfi) &= e(\bfi)y_r;\\
\psi_r e(\bfi) &= e(s_r\bfi) \psi_r;\\
y_ry_s &= y_sy_r;\\
\psi_r y_s &= \mathrlap{y_s\psi_r}\hphantom{\smash{\begin{cases}(\psi_{r+1}\psi_r\psi_{r+1}+1)e(\bfi)\\\\\\\end{cases}}}\kern-\nulldelimiterspace\text{if } s\neq r,r+1;\\
\psi_r \psi_s &= \mathrlap{\psi_s\psi_r}\hphantom{\smash{\begin{cases}(\psi_{r+1}\psi_r\psi_{r+1}+1)e(\bfi)\\\\\\\end{cases}}}\kern-\nulldelimiterspace\text{if } |r-s|>1;\\
y_r \psi_r e(\bfi) &= (\psi_r y_{r+1} - \delta_{i_r,i_{r+1}})e(\bfi);\\
y_{r+1} \psi_r e(\bfi) &= (\psi_r y_r + \delta_{i_r,i_{r+1}})e(\bfi);\\
\psi_r^2 e(\bfi)&=\begin{cases}
\mathrlap0\phantom{(\psi_{r+1}\psi_r\psi_{r+1}+1)e(\bfi)}& \text{if }i_r=i_{r+1},\\
e(\bfi) & \text{if }i_{r+1}\neq i_r, i_r\pm1,\\
(y_{r+1} - y_r) e(\bfi) & \text{if }i_r = i_{r+1} + 1,\\
(y_r - y_{r+1}) e(\bfi) & \text{if }i_r = i_{r+1} - 1;
\end{cases}\\
\psi_r\psi_{r+1}\psi_re(\bfi) &= \begin{cases}
(\psi_{r+1}\psi_r\psi_{r+1}+1)e(\bfi)& \text{if }i_{r+2} = i_r = i_{r+1} +1,\\
(\psi_{r+1}\psi_r\psi_{r+1}-1)e(\bfi)& \text{if }i_{r+2} = i_r = i_{r+1} - 1,\\
(\psi_{r+1}\psi_r\psi_{r+1})e(\bfi)& \text{otherwise;}
\end{cases}
\end{align*}
}for all admissible $r, s, \bfi, \bfj$.
When $e= 2$, we actually have slightly different `quadratic' and `braid' relations, which may be found, for example, in \cite[\S 3.1]{kmr}.
We omit them here, as we will not explicitly calculate with these relations when $e=2$.

Recalling that $\La = \La_\kappa$, we define $\mathscr{R}_\alpha^\La$ by imposing one additional relation on \(\mathscr{R}_\alpha\):
\begin{align*}
y_1^{(\La, \alpha_{i_1})} e(\bfi) &= 0,
\end{align*}
for all \(\bfi \in I^\alpha\).

\begin{lemc}{bk09}{Corollary 1}
There is a unique $\bbz$-grading on $\mathscr{R}_\alpha^{\Lambda}$ such that, for all admissible $r$ and $\bfi$,
\[
\deg(e(\bfi))=0,\quad\deg(y_r)=2,\quad\deg\psi_r(e(\bfi))=-a_{i_r,r_{r+1}}.
\]
\end{lemc}

The \emph{cyclotomic KLR algebra} or \emph{cyclotomic quiver Hecke algebra} $\mathscr{R}_n^\La$ is defined to be the direct sum $\bigoplus_\alpha\mathscr{R}_\alpha^\La$, where the sum is taken over all $\alpha\in Q^+$ of height $n$.

These $\bbz$-graded algebras are connected to the Hecke algebras of type $B$ via (a special case of) Brundan and Kleshchev's \emph{Graded Isomorphism Theorem}.

\begin{thmc}{bkisom}{Main Theorem}
If $e = \nchar(\bbf)$ or $\nchar(\bbf)\nmid{e}$, then $\mathscr{R}_n^\La$ is isomorphic to the integral Hecke algebra $\hhh(\xi,Q_1,Q_2)$ of type $B$ with parameters $\xi \in \bbf$ a primitive $e$th root of unity, $Q_1 = \xi^{\kappa_1}$, and $Q_2 = \xi^{\kappa_2}$.
That is, $\hhh(\xi,Q_1,Q_2)$ has generators $T_0,\dots, T_{n-1}$ satisfying type $B$ Coxeter relations, with the quadratic relations replaced with
\[
(T_0 - \xi^{\kappa_1}) (T_0 - \xi^{\kappa_2}) = 0 \quad \text{and} \quad (T_i-\xi) (T_i+1) = 0 \;\; \forall \; 1\leq i \leq n-1.
\]
\end{thmc}

Finally, we end the subsection by briefly recalling the sign isomorphism $\operatorname{sgn}:\scrr_n^\kappa \rightarrow \scrr_n^{\kappa'}$ from~\cite[\S3.3]{kmr}, where $\kappa'= (-\kappa_2,-\kappa_1)$.
It is the homogeneous algebra isomorphism defined by
\[
\operatorname{sgn}: e(i_1,i_2,\dots,i_n)\mapsto e(-i_1,-i_2,\dots,-i_n), \; y_r \mapsto -y_r, \; \psi_s \mapsto -\psi_s, \; \text{for all } (i_1,i_2,\dots,i_n), r, s.
\]

\subsection{Basic representation theory of KLR algebras}

We begin by recalling some necessary definitions pertaining to graded modules.
If $M$ is any graded $\scrr_n$- (or $\scrr_n^\La$-)module, we let $M\langle r \rangle$ denote the graded module obtained by shifting the grading on $M$ up by $r$; that is, $M\langle r \rangle_d = M_{d-r}$.
We let $M^{\operatorname{sgn}}$ denote the same graded vector space as $M$, with the action of $\scrr_n$ twisted by the sign isomorphism.
For $q$ an indeterminate, the Grothendieck groups of $\scrr_n$ and $\scrr_n^\La$ are $\bbz[q,q^{-1}]$-modules by letting $q^m$ act by a degree shift by $m$.
We let $\underline{M}$ denote the (ungraded) module obtained by forgetting the grading on $M$.
The (graded) dual of $M$ is $M^\circledast := \Hom_\bbf(M,\bbf)$ with $\scrr_n^\La$-action given by $(h\cdot f)(m) = f(\tau(h)m)$ for $f \in M^\circledast$, $m\in M$, $h \in \scrr_n^\La$.
Here, $\tau$ denotes the homogeneous algebra anti-involution of $\scrr_n^\La$ that sends each generator to themselves.

For any graded \(\scrr_\alpha\)-module, and \(\bfi \in I^\alpha\), we write \(M_\bfi := e(\bfi)M\), giving a vector space decomposition \(M = \bigoplus_{\bfi \in I^\alpha} M_\bfi\). If \(M_\bfi \neq 0\), we say that \(\bfi\) is a {\em word} of \(M\). 
The \emph{formal $q$-character} of $M$ is
\[
\operatorname{ch}_q M := \sum_{\bfi \in I^\alpha} \left( \operatorname{qdim} M_{\bfi} \right) \cdot \bfi,
\]
where $\operatorname{qdim} M \in \mathbb{Z}[q,q^{-1}]$ is the \emph{graded dimension} of $M_{\bfi}$. We have \(\operatorname{ch}_q M^\circledast = \overline{\operatorname{ch}_q M}\), where the bar indicates the {\em bar involution}, i.e. the automorphism of \(\mathbb{Z}[q,q^{-1}]\) which swaps \(q\) and \(q^{-1}\), extended linearly to \(\bigoplus_{i \in I^\alpha} \mathbb{Z}[q,q^{-1}] \bfi\).

There is an induction functor $\operatorname{Ind}_{\alpha,\beta}$, which associates to an $\scrr_\alpha$-module $M$ and an $\scrr_\beta$-module $N$ the $\scrr_{\alpha+\beta}$-module $M\circ N := \operatorname{Ind}_{\alpha,\beta} M\boxtimes N $ for $\alpha,\beta\in Q_+$.

\subsection{Specht modules}
The graded $\scrr_n^\La$-modules of prime importance to us are the Specht modules.
We use the following presentation of Specht modules from~\cite[Definition 7.11]{kmr}; we give only the presentations for the two families of bipartitions that we will explicitly need.

\begin{defn}\label{defn:regularspecht}
	Let $\la = ((ke,je-e+1),(e-1)) \in \mptn 2 n$ for some $j,k\in\bbn$.
	The (column) \emph{Specht module} $\spe\la$ is the cyclic $\scrr_n^\La$-module generated by $z_\la$ of degree $\deg(z_\la):= \codeg(\ttt_\la)$ subject to the relations:
	\begin{itemize}
		\item $e(\bfi_\la)z_\la = z_\la$;
		
		\item $y_r z_\la = 0$ for all $r \in \{1, \dots, n\}$;
		
		\item $\psi_r z_\la=0$ for all $r\in\{1, \dots,e-2\} \cup \{e,e+2,e+4,\dots,2je-e\}\cup\{2je-e+2,\dots,ke+je-1\}$;
		
		
		\item $\psi_r\psi_{r+1} z_\la=0$ for all $r\in\{e,e+2,e+4,\dots,2je-e\}$;
		
		\item $\psi_r\psi_{r-1} z_\la=0$ for all $r\in\{e+2,e+4,\dots,2je-e\}$.
		
	\end{itemize}
	
	Let $\mu = ((ke),(je)) \in \mptn 2 n$ for some $j,k \in\bbn$.
	The (column) \emph{Specht module} $\spe\mu$ is the cyclic $\scrr_n^\La$-module generated by $z_\mu$ of degree $\deg(z_\mu):= \codeg(\ttt_\mu)$ subject to the relations:
	\begin{itemize}
		\item $e(\bfi_\mu)z_\mu = z_\mu$;
		
		\item $y_r z_\mu = 0$ for all $r \in \{1, \dots, n\}$;
		
		\item $\psi_r z_\mu = 0$ for all $r\in \{1,\dots,n-1\}\setminus \{je\}$.	\end{itemize}
\end{defn}

For each $w\in\sss$, we fix a reduced expression $w = s_{i_1} \dots s_{i_m}$ throughout.
We define the associated element of $\mathscr{R}_n^{\Lambda}$ to be $\psi_w:=\psi_{i_1}\dots{\psi_{i_m}}$, which, in general, depends on the choice of reduced expression for $w$.
For $\la \in \mptn 2 n$ and a $\la$-tableau $\ttt$, we define $v_{\ttt}:=\psi_{w_{\ttt}} z_\la$.
Whilst these vectors $v_{\ttt}$ of $\spe\la$ also depend on the choice of reduced expression in general, the following result does not.

\begin{thm}\textbf{\textup{\cite[Corollary 4.6]{bkw11}, \cite[Proposition 7.14 and Corollary 7.20]{kmr}}}\label{thm:basistab}
	For $\la \in \mptn 2 n$, the set of vectors $\{v_{\ttt}\mid{\ttt\in\std\la}\}$ is a homogeneous $\bbf$-basis of $\spe\la$, with $\deg(v_{\ttt}) = \codeg(\ttt)$.
\end{thm}

\begin{Assumption}
In light of \cref{thm:decompbihooks}, whose Specht modules are those of interest in the present paper, we will assume that $\kappa=(0,0)$ throughout the remainder of the paper.
\end{Assumption}

We will compute homomorphisms of Specht modules, for which the following lemma will be useful.

\begin{lemc}{bkw11}{Lemma 4.4}\label{lem:tabresi}
	Let $\la \in \mptn 2 n$, and $\ttt \in \std\la$.
	Then $e(\bfi) v_\ttt = \delta_{\bfi,\bfi_\ttt} v_\ttt$.
\end{lemc}

The next lemma serves as a reduction result for analysing our Specht modules.

\begin{lem}\label{lem:switchcomps}
Suppose that $k\geq j$.
Then $\spe{((je),(ke))} \cong \spe{((ke),(je))}^\circledast \langle j+k \rangle$.
\end{lem}

\begin{proof}
By examining the presentations, and noting that $\codeg (\ttt_{((je),(ke))}) = k$ while $\deg (\ttt^{((ke),(je))}) = j+2k$, we see that
\[
\spe{((je),(ke))} \langle j+k \rangle \cong \rspe{((ke),(je))}.
\]
The definitions of the `row-initial tableau' $\ttt^{((ke),(je))}$ and the `row Specht module' $\rspe{((ke),(je))}$ appearing here may be found in \cite[Sections~2.3~and~5.4]{kmr}. 
Next, \cite[Theorem 7.25]{kmr} and an easy degree calculation gives us that
\begin{equation}\label{eq:duals}
\rspe{((ke),(je))} \cong \spe{((ke),(je))}^\circledast \langle 2j+2k\rangle,
\end{equation}
and hence that
\[
\spe{((je),(ke))} \cong \spe{((ke),(je))}^\circledast \langle j+k \rangle.\qedhere
\]
\end{proof}

\subsection{Regular bipartitions and simple modules}

Let $\la \in \mptn 2 n$.
We define the \emph{$i$-signature of $\la$} by reading the Young digram $[\la]$ from the top of the first component down to the bottom of the last component, writing a $+$ for each addable $i$-node and a $-$ for each removable $i$-node.
We obtain the \emph{reduced $i$-signature of $\la$} by successively deleting all adjacent pairs $+-$ from the $i$-signature of $\la$, always of the form $-\dots-+\dots+$.

The removable $i$-nodes corresponding to the $-$ signs in the reduced $i$-signature of $\la$ are called the \emph{normal} $i$-nodes of $\la$, while the addable $i$-nodes corresponding to the $+$ signs in the reduced $i$-signature of $\la$ are called the \emph{conormal} $i$-nodes of $\la$.
The lowest normal $i$-node of $[\la]$, if there is one, is called the \emph{good} $i$-node of $\la$, which corresponds to the last $-$ sign in the $i$-signature of $\la$.
Analogously, the highest conormal $i$-node of $[\la]$, if there is one, is called the \emph{cogood} $i$-node of $\la$, which corresponds to the first $+$ sign in the $i$-signature of $\la$.

We say that a bipartition $\la \in \mptn 2 n$ is \emph{regular}, or \emph{conjugate-Kleshchev}, if $[\la]$ can be obtained by successively adding cogood nodes to $\varnothing$.
That is, we have a sequence $\varnothing = \la(0), \la(1), \dots, \la(n) = \la$ such that $[\la(i)] \cup\{A\} = [\la(i+1)]$, where $A$ is a cogood node of $\la(i)$.
Equivalently, $\la$ is regular if and only if $\varnothing$ can be obtained by successively removing good nodes from $[\la]$.
Observe in level one that the set of all regular partitions coincides with the set of all $e$-regular partitions.

For each regular bipartition $\la \in \mptn 2 n$, the Specht module $\spe\la$ has a simple head, denoted by $\D\la$.

\begin{thmc}{bk09}{Theorem 5.10}
The modules $\{\D\la \mid \la \in \mptn 2 n, \ \la \text{ regular}\}$ give a complete set of graded simple $\scrr_n^\La$-modules up to isomorphism and grading shift.
Moreover, each $\D\la$ is self-dual as a graded module.
\end{thmc}

There is a bijection $m_{e,\kappa}:\mptn 2 n \rightarrow \mptn 2 n$ such that $(\D\la)^{\operatorname{sgn}} \cong \D{m_{e,\kappa}(\la)}$.
Strictly speaking, the sign map actually takes a simple $\scrr^{\La_\kappa}_n$-module to a $\scrr^{\La_{-\kappa}}_n$-module, but we may ignore this here as we only consider $\kappa = (0,0)$, so that the sign map really permutes simple $\scrr^{\La_\kappa}_n$-modules.
This higher level analogue of the Mullineux involution was introduced by Fayers in~\cite[\S2]{fay08} (see also~\cite{jl09}).
More precisely, if $\la \in \mptn \kappa n$ is obtained from $\varnothing$ by consecutively adding cogood nodes with residues $i_1, i_2, \dots, i_n$, respectively, then $m_{e,\kappa}(\la)$ is the multipartition obtained from $\varnothing$ by consecutively adding cogood nodes of residues $-i_1, -i_2, \dots, -i_n$, respectively.

\subsection{Divided power functors}\label{subsec:divpowers}

Let $\varphi_i(\la)$ denote the number of conormal $i$-nodes of a multipartition $\la$, and $\tilde {f_i} \la$ the partition obtained from $\la$ by adding the cogood $i$-node.
In \cite[Sections 4.4 and 4.6]{bk09}, graded induction functors and graded divided power functors are introduced.
Here, we denote these by $f_i$ and $f_i^{(n)}$, respectively.
We refer to \cite{bk09} for further details, and here only mention some key facts that we will require later.

For a non-negative integer $n$, we define the \emph{quantum integer} $[n]:= q^{-(n-1)} + q^{-(n-3)} +\dots + q^{n-1}$ and the \emph{quantum factorial} $[n]! := [1] [2] \dots [n]$.

\begin{lemc}{bk09}{Lemma 4.8, Theorem 4.12}\label{lem:grbranch}
There is an isomorphism $f_i^{n} \cong [n]! f_i^{(n)}$.
Moreover, $f_i \D\la$ is non-zero if and only if $\varphi_i(\la) \neq 0$, in which case $f_i \D\la$ has irreducible socle isomorphic to $\D{\tilde {f_i} \la}\langle \varphi_i(\la) - 1 \rangle$ and head isomorphic to $\D{\tilde {f_i} \la}\langle 1 - \varphi_i(\la) \rangle$.
\end{lemc}

\subsection{Irreducible Weyl modules over the Schur algebra}\label{sec:irrSpechts}

It will be useful for us to know which Weyl modules (over the usual Schur algebra -- i.e.~level 1,  with $e=p$) $\Delta(\la)$ indexed by two-column partitions are irreducible.
In such cases, $\Delta(\la)$ is equal to its simple head, $\sL{\la}$.
Here, by Weyl modules, we mean the Weyl modules used in \cite{Mathas}.
The Schur functor maps these modules to `row Specht modules'.
Because of the distinction between level 1 and level 2 cases, this will not be problematic for us.
For $(a,b) \in [\la]$, we denote by $h^\la(a,b)$ its hook length, i.e.~$h^\la(a,b) = \la_a - b + \la_b' - a + 1$.
For a prime $p$, define $\nu_p(h)$ to be the largest power of $p$ dividing $h$.

\begin{thmc}{Mathas}{Theorem 5.42}\label{thm:irredweyl}
Let $\la \vdash n$, and let $\bbf$ be a field of characteristic $p$.
The Weyl module $\Delta(\la)$ over the Schur algebra $S(n,n)$ is irreducible if and only if $\nu_p(h^\la(a,b)) = \nu_p(h^\la(a,c))$ for every $(a,b), (a,c) \in [\la]$.
\end{thmc}

\begin{cor}\label{cor:irredweyls}
Let $e=p$.
\begin{enumerate}[label=(\roman*)]
	\item If $p\neq 2$ and $n\geq 2j$, the Weyl modules $\Delta{(1^{n})}, \Delta{(2,1^{n-2})}, \dots, \Delta{(2^j,1^{n-2j})}$ are all simultaneously irreducible if and only if $p$ does not divide any of the integers $n, n-1, \dots, n-2j+2$.
	
	\item If $p=2$, then the Weyl modules:
		\begin{enumerate}
			\item $\Delta{(1^{n})}$ and $\Delta{(2,1^{n-2})}$ are simultaneously irreducible if and only if $n$ is odd;
			
			\item $\Delta{(1^{n})}$ and $\Delta{(2,1^{n-2})}$, and $\Delta{(2^2,1^{n-4})}$ are simultaneously irreducible if and only if $n\equiv 3 \pmod 4$;
			
			\item $\Delta{(1^{n})}$ and $\Delta{(2,1^{n-2})}$, $\Delta{(2^2,1^{n-4})}$, and $\Delta{(2^3,1^{n-6})}$ are never simultaneously irreducible.
		\end{enumerate}
	
\end{enumerate}
\end{cor}

\begin{proof}
We proceed by induction on $j$, with some special care needed for small $j$.
If $j=1$, we want to know when $\Delta(1^{n})$ and $\Delta(2,1^{n-2})$ are simultaneously irreducible.
The former module is always irreducible, and the latter is irreducible if and only if $p\nmid n$.
If $p\mid n$, then $\Delta(2,1^{n-2})$ is comprised of a simple head $\sL{2,1^{n-2}}$, and a simple socle $\sL{1^{n}}$.

If $j=2$, we use the fact that $\Delta(1^{n})$ and $\Delta(2,1^{n-2})$ are simultaneously irreducible if and only if $p\nmid n$, and check when $\Delta(2^2,1^{n-4})$ is also irreducible.
We apply \cref{thm:irredweyl}, noting that we only need to check the valuations of hook lengths in the first two rows.
The nodes $(1,2)$ and $(2,2)$ of $[(2^2,1^{n-4})]$ have hook lengths $2$ and $1$, respectively.
In characteristic $p\neq 2$, both give $p$-adic valuations 0.
The nodes $(1,1)$ and $(2,1)$ have hook lengths $n-1$ and $n-2$, respectively.
Thus if $p\neq2$, $\Delta(2^2,1^{n-4})$ is irreducible if and only if $p\nmid n-1,n-2$.
If $p=2$, then $\nu_p(h^\la(1,2)) = 1$, so one easily checks that $\Delta(2^2,1^{n-4})$ is irreducible if and only if $n\equiv 3 \pmod 4$.
If $n\equiv 3 \pmod 4$, then one can easily see that $\Delta{(2^3,1^{n-6})}$ is reducible, completing the proof when $p=2$.

We now assume that $p\neq 2$, $j>2$, and that the Weyl modules $\Delta{(1^{n})}, \Delta{(2,1^{n-2})}, \dots, \Delta{(2^{j-1},1^{n-2j+2})}$ are all simultaneously irreducible if and only if $p$ does not divide any of the integers $n, n-1, \dots, n-2j+4$.
In particular, since these are $2j-3$ consecutive integers, we can assume that if the aforementioned Weyl modules are all irreducible, then $p>2j-3$.
We need to check when $\Delta{(2^j,1^{n-2j})}$ is irreducible.
Since $p>2j-3$, the $p$-adic valuations in the second column of $[(2^j,1^{n-2j})]$ are all $0$.
The hook lengths in nodes $(1,1),(2,1),\dots,(j,1)$ are $n-j+1, n-j, \dots, n-2j+2$, respectively, and the result follows from \cref{thm:irredweyl}.
\end{proof}

Next, we introduce a modular version of the Pieri rule.
James~\cite{j77} proves a more general version of this (the Littlewood--Richardson rule) for the symmetric group, and outlines how a version for Weyl modules over the general linear group follows.
Since these Weyl modules descend to our Weyl modules for the Schur algebra, we thus have the result below.

\begin{lemc}{j77}{Sections 10 \& 11}\label{lem:Pierifilt}
Let $k, j\geq 1$.
The tensor product of Weyl modules $\Delta(1^k) \otimes \Delta(1^j)$ has a filtration by the Weyl modules $\Delta{(1^{k+j})}, \Delta{(2,1^{k+j-2})}, \dots, \Delta{(2^j,1^{k-j})}$.
\end{lemc}

We end this subsection with a result computing the composition factors of Weyl modules indexed by two-column partitions.
This result will be useful to us in \cref{sec:main}.
First we must introduce some notation.
Suppose that $a$ and $b$ are positive integers with $p$-adic expansions
\[
a = a_0 + p a_1 + p^2 a_2 + \cdots \quad \text{ and } \quad b = b_0 + p b_1 + p^2 b_2 + \cdots
\]
with $0 \leq a_i, b_i < p$ for all $i$.
Write $a \preccurlyeq_p b$ if, for all $i$ either $a_i = 0$ or $a_i = b_i$.

\begin{thmc}{Mathas}{Section~6.4, Rule 15}\label{thm:decompnos}
Let $\la = (2^{m},1^{n-2m})$ and $\mu = (2^{j},1^{n-2j})$.
Then
\[
[\Delta(\la):\sL\mu] = \begin{cases}
1 & \text{if } \lfloor \frac{m-j}{p} \rfloor \preccurlyeq_p \lfloor \frac{n-2j+1}{p} \rfloor  \text{ and either } p\mid m{-}j \text{ or } p\mid n{-}m{-}j{+}1,\\
0 & \text{otherwise.} 
\end{cases}
\]
\end{thmc}

\section{Imaginary Specht modules}\label{sec:imag}

\subsection{The imaginary Schur algebra}\label{imagSchSec}

Here we import some useful results from~\cite{km17}.
Let \(\succcurlyeq\) be a balanced convex preorder on \(\Phi_+\) (see for instance \cite[\S3.1]{klesh14}).
Recall from \cref{subsec:Lie} that $\delta = \alpha_0 + \dots + \alpha_{e-1}$ denotes the null root.
By \cite[Lemma 5.1, Corollary 5.3]{klesh14}, there exists a unique simple \(\mathscr{R}_\delta\)-module \(\operatorname{L}_{\delta, e-1}\) such that for all words \(\bfj = (j_1, \ldots, j_e)\) in \(\operatorname{L}_{\delta, e-1}\), we have \(j_1 = 0\) and \(j_e = e-1\).
The Specht module \(\spe{(e)}\) is a simple one-dimensional \(\scrr_\delta\)-module with character \(\bfi = (0,1,2, \ldots, e-1)\).
Thus it follows that \(\spe{(e)} \cong \operatorname{L}_{\delta, e-1}\).

For \(\nu = (n_1, \ldots, n_a) \in \mathscr{C}_n(a)\) define the associated {\it Gelfand--Graev words}  in \(I^{n \delta}\) via
\[
\bgg^{(n)}=0^n1^n \dots (e-1)^n,
\qquad \qquad
\bgg^\nu = \bgg^{(n_1)} \dots \bgg^{(n_a)}.
\]
In this notation, we collect identical terms so that for example \(0^n\) is used in place of a length $n$ string of zero residues.
We also write \(\bfi^k\) to mean the concatenation of \(k\) copies of \(\bfi\).
We also define
\[
c(\nu) = ([n_1]! \dots [n_a]!)^e.
\]

Following \cite{km17}, for \(n \in \bbn\), denote the \(\scrr_{n \delta}\)-module
\(
M_n := \operatorname{L}_{\delta, e-1}^{\circ n},
\)
the {\it imaginary tensor space module (of colour \(e{-}1\))} and the finite-dimensional quotient algebra \(\mathscr{S}_n := \scrr_{n \delta} / \textup{Ann}_{\scrr_{n \delta}} M_n\) the {\it imaginary Schur algebra}. This quotient algebra earns its name via a Morita equivalence with the classical Schur algebra \(S_\bbf(n,n)\).
We view the Schur algebra \(S_\bbf(n,n)\) as a graded algebra concentrated in degree \(0\), and thus let \(S_\bbf(n,n)\textup{-mod}\) denote the category of \emph{graded} \(S_\bbf(n,n)\)-modules.
We will freely consider \(\mathscr{S}_n\)-modules as \(\scrr_{n\delta}\)-modules via inflation.
\cref{MainKMimag} below is a special case of \cite[Theorems 4 \& 5]{km17}.

\subsection{Imaginary Schur--Weyl duality}\label{imagSWdSec}

\begin{thmc}{km17}{Theorems 4 \& 5}\label{MainKMimag}
Let \(h \geq n \in \bbn\).
There exists a projective generator \(Z\) for \(\mathscr{S}_n\) such that \(\End_{\mathscr{S}_n}(Z) \cong S_\bbf(h,n)\), and considering \(Z\) as a right \(S_\bbf(h,n)\)-module, we have \(\End_{S_\bbf(h,n)}(Z) \cong \mathscr{S}_n\).
This gives mutually inverse equivalences of categories:
\begin{align*}
&\scrm_{h,n} : \mathscr{S}_n\textup{-mod} \to S_\bbf(h,n)\textup{-mod}, \qquad V \mapsto \Hom_{\mathscr{S}_n}(Z,V)\\
&\widehat\scrm_{h,n} : S_\bbf(h,n)\textup{-mod} \to \mathscr{S}_n\textup{-mod}, \qquad W \mapsto Z \otimes_{S_\bbf(h,n)} W.
\end{align*}
This equivalence intertwines induction in the imaginary Schur algebra and the tensor product in the classical Schur algebra, in the following sense.
Let \(\nu = (n_1, \ldots, n_a) \in  \mathscr{C}_n(a)\).
We have an isomorphism of functors:
\begin{align*}
\textup{Ind}_{n_1\delta, \ldots, n_a\delta}^{n\delta} ( \widehat\scrm_{h,n_1}( ?) \boxtimes \cdots \boxtimes \widehat\scrm_{h,n_a}(?))
\cong
\widehat\scrm_{h,n} ( ? \otimes \cdots \otimes ?)
\end{align*}
from \(S_\bbf(h,n_1)\textup{-mod} \times \cdots \times S_\bbf(h,n_a)\textup{-mod}\) to \( \mathscr{S}_n\textup{-mod}\).
\end{thmc}

In this paper we use the Morita equivalence above as a black box, and will not need to be concerned with the specifics of the projective generator \(Z\).
We will write \(\scrm_n\) for \(\scrm_{n,n}\), and \(\scrm\) for \(\scrm_n\) when \(n\) is clear from context, and similarly for \(\widehat\scrm\).

Let \(h \geq n\). Recall (see \cite{Donkin}) that \(S_\bbf(h,n)\) is a finite-dimensional quasi-hereditary algebra with irreducible, standard, and indecomposable tilting modules
\begin{align*}
\operatorname{L}^h(\la), \;\; \Delta^h(\la), \;\; \operatorname{T}^h(\la), \qquad (\la \in \scrp_n).
\end{align*}
We will omit the superscript \(h\) in the situation that \(h =n\).

Via \cref{MainKMimag} we have that \(\mathscr{S}_n\) is itself a finite-dimensional quasi-hereditary algebra with irreducible, standard, and indecomposable tilting modules
\begin{align*}
\widehat\scrm(\operatorname{L}(\la)), \;\; \widehat\scrm(\Delta(\la)), \;\;\widehat\scrm(\operatorname{T}(\la)), \qquad (\la \in \mathscr{P}_n).
\end{align*}

\begin{lemc}{km17}{Lemma 6.1.3}\label{hirrev}
Let \(h\geq n\), and \(\la \in \mathscr{P}_n\). Then we have \(\widehat\scrm_{h,n}(\operatorname{L}^h(\la)) \cong \widehat\scrm(\sL{\la}) \) and  \(\widehat\scrm_{h,n}(\Delta^h(\la)) \cong \widehat\scrm(\Delta(\la)) \).
\end{lemc}

\begin{lemc}{km17}{Lemma 6.3.2}\label{dualsame}
Let \(\la \in \mathscr{P}_n\).
Then we have that \( \widehat{\scrm}(\operatorname{L}(\la))^\circledast \cong \widehat{\scrm}( \operatorname{L}(\la)) \).
\end{lemc}

Since the characters of images of simple \(S_{\bbf}(n,n)\)-modules under \(\widehat{\scrm}\) are bar-invariant, we have the following immediate result.

\begin{cor}\label{barimage}
Let \(W \in S_\bbf(n,n)\textup{-mod}\). Then 
\begin{align*}
\operatorname{ch}_q \widehat{\scrm}(W) = \overline{\operatorname{ch}_q \widehat{\scrm}(W)} = \operatorname{ch}_q \widehat{\scrm}(W)^\circledast.
\end{align*}
\end{cor}

Knowledge of the usual formal character \(\operatorname{ch} W\) of a module \(W\) over the classical Schur algebra \(S_\bbf(n,n)\) (i.e., the dimension of the usual weight spaces \(W_\la = e(\la) W\) for \(\la \in \mathscr{C}_n(n)\), provides partial information about the graded character of \(\scrm(W)\):

\begin{thmc}{km17}{Theorem 9}\label{KMWords1}
Let \(\la \in \mathscr{C}_n(n)\) and \(W \in S_\bbf(n,n)\textup{-mod}\). Then
\begin{align*}
\operatorname{qdim} \widehat{\scrm}(W)_{\bgg^\la} = c(\la) \dim W_\la.
\end{align*}
\end{thmc}

For \(\la \in \mathscr{P}_n, \mu \in \mathscr{C}_n(n)\), let \(K_{\la, \mu}\) denote the usual Kostka number; the dimension of the \(\mu\)-weight space \(\Delta(\la)_\mu\) in the standard module \(\Delta(\la)\), given by the number of semistandard \(\la\)-tableaux of weight \(\mu\).
Let \(k_{\la, \mu}\) denote the dimension of the \(\mu\)-weight space \(\sL{\la}_\mu\) in the simple module \(\sL{\la}\).
If \(\textup{char} \;\bbf = 0\), then \(K_{\la,\mu} = k_{\la, \mu}\).
The following theorem, a special case of \cite[Theorem 10]{km17}, provides important partial information about the character of these modules:

\begin{thmc}{km17}{Theorem 10}\label{KMwords}
For \(\la \in \mathscr{P}_n, \mu \in \mathscr{C}_n(n)\), we have
\begin{align*}
\dim \widehat\scrm(\Delta(\la))_{\bgg^\mu} = c(\mu) K_{\la, \mu}
\qquad
\textup{and}
\qquad
\dim \widehat\scrm(\sL{\la})_{\bgg^\mu} = c(\mu) k_{\la, \mu}.
\end{align*}
\end{thmc}

\begin{cor}\label{wordrec}
If \(\operatorname{L}\) is a simple \(\mathscr{S}_n\)-module, and there exists \(\la \in \mathscr{P}_n\) such that \(\dim \operatorname{L}_{\bgg^\la} >0\) and \(\dim \operatorname{L}_{\bgg^\mu} =0\) for all \(\mu \in \mathscr{P}_n\) with \(\mu \doms \la\), then \(\operatorname{L} \cong \widehat\scrm(\sL{\la})\).
\end{cor}
\begin{proof}
Simple modules for \(S_\bbf(n,n)\) are distinguished by their highest weights (see for instance \cite[Theorem 3.5a]{Green}), so if \(\operatorname{L}'\) is a simple \(S_\bbf(n,n)\)-module such that \(\dim \operatorname{L}'_\la >0\) and \(\dim \operatorname{L}'_\mu\) = 0 for all \(\mu \doms \la\), it follows that \(\operatorname{L}' \cong \sL{\la}\).
Then \cref{KMwords} implies the result.
\end{proof}

We now show that the self-duality of tilting modules is preserved by \(\widehat{\scrm}\):
\begin{prop}\label{tiltdual}
Let \(\la \in \mathscr{P}_n\).
Then we have \(\widehat{\scrm}(\operatorname{T}(\la))^\circledast \cong \widehat{\scrm}(\operatorname{T}(\la))\).
\end{prop}
\begin{proof}
By \cite[3.3(1)]{Donkin}, the tilting module \(\operatorname{T}(\la)\) arises as an indecomposable summand of \(W = \operatorname{L}^n(1^{\mu_1}) \otimes \cdots \otimes \operatorname{L}^n(1^{\mu_n})\) for some \(\mu \in \mathscr{P}_n\), and moreover, all indecomposable summands of \(W\) are tilting modules \(T(\nu)\), where \(\nu \in \mathscr{P}_n\).
Now we note that
\begin{align*}
\widehat{\scrm}(W)^\circledast &= \widehat{\scrm}(\operatorname{L}^n(1^{\mu_1}) \otimes \cdots \otimes \operatorname{L}^n(1^{\mu_n}) )^\circledast 
\cong
( \widehat{\scrm}(\operatorname{L}^n(1^{\mu_1})) \circ \cdots \circ \widehat{\scrm}(\operatorname{L}^n(1^{\mu_n})) )^\circledast\\
&\cong
 \widehat{\scrm}(\operatorname{L}(1^{\mu_n}))^\circledast \circ \cdots \circ \widehat{\scrm}(\operatorname{L}(1^{\mu_1}))^\circledast
\cong
 \widehat{\scrm}(\operatorname{L}(1^{\mu_n})) \circ \cdots \circ \widehat{\scrm}(\operatorname{L}(1^{\mu_1}))\\
 &\cong
  \widehat{\scrm}(\operatorname{L}^n(1^{\mu_1})) \circ \cdots \circ \widehat{\scrm}(\operatorname{L}^n(1^{\mu_n}))
  \cong
   \widehat{\scrm}(\operatorname{L}^n(1^{\mu_1}) \otimes \cdots \otimes \operatorname{L}^n(1^{\mu_n}) )
   =
  \widehat{\scrm}(W),
\end{align*}
where the second isomorphism follows from \cref{hirrev} and \cite[Lemma 3.4.2]{km17}, the third isomorphism follows from \cref{dualsame}, and the fourth isomorphism follows from \cref{hirrev} and \cite[Lemma 6.2.2]{km17}. Since \(\widehat{\scrm}(W)^\circledast \cong \widehat{\scrm}(W)\), it follows that we must have \(\widehat{\scrm}(\operatorname{T}(\la))\cong \widehat{\scrm}(\operatorname{T}(\nu))^\circledast \) for some tilting module \(T(\nu)\) which is an indecomposable summand of \(W\). But then by \cref{barimage} we have
\begin{align*}
\operatorname{ch}_q \widehat{\scrm}(\operatorname{T}(\la)) = 
\operatorname{ch}_q \widehat{\scrm}(\operatorname{T}(\nu))^\circledast = 
\operatorname{ch}_q \widehat{\scrm}(\operatorname{T}(\nu)).
\end{align*}
Then by \cref{KMWords1} we have that \(\operatorname{ch}\operatorname{T}(\la)= \operatorname{ch}\operatorname{T}(\nu)\). Then, as noted in \cite[Remark 3.3(i)]{Donkin} we must have \(\la = \nu\), so the result follows.
\end{proof}

\subsection{Imaginary Specht modules}

\begin{lem}\label{Ske}
For \(k \in \bbn\), we have \(\spe{(ke)} \cong \widehat\scrm(\Delta(1^k)) \cong \widehat\scrm(\sL{1^k})\) as \(\scrr_{k \delta}\)-modules.
\end{lem}

\begin{proof}
We have 
\(\textup{Res}_{\delta, \ldots, \delta} \spe{(ke)} \cong V_1 \boxtimes \cdots \boxtimes V_k\)
for some \(\scrr_{\delta}\)-modules \(V_1, \ldots, V_k\). But, since 
\(\spe{(ke)}\) is a simple one-dimensional module with character \(\bgg^{(1^k)} = \bfi^k\), it follows that \(V_1, \ldots, V_k\) each have character \(\bfi\), so by consideration of \S\ref{imagSchSec}, \(V_i \cong  \operatorname{L}_{\delta, e-1}\) for \(i = 1, \ldots, k\). 
Thus \(\textup{Res}_{\delta, \ldots, \delta} \spe{(ke)} \cong \operatorname{L}_{\delta, e-1}^{\boxtimes k}\), so by reciprocity there is a nonzero \(\scrr_{k\delta}\)-homomorphism from \(M_k \cong \operatorname{L}_{\delta, e-1}^{\circ k}\) to \(\spe{(ke)}\). As \(\spe{(ke)}\) is simple, this map is a surjection. Thus  
\((\textup{Ann}_{\scrr_{k \delta}} M_k) \spe{(ke)} = 0\), and so we have a well-defined action of  \(\mathscr{S}_k = \scrr_{k \delta} / \textup{Ann}_{\scrr_{k \delta}} M_k\) on \(\spe{(ke)}\). It follows by \cref{wordrec} that \(\spe{(ke)} \cong \widehat\scrm(\sL{1^k})\), and \(\widehat\scrm(\Delta(1^k)) \cong \widehat\scrm(\sL{1^k})\) as \(\mathscr{S}_k\)-modules, since \(\Delta(1^k) = \sL{1^k}\) as  \(S_\bbf(k,k)\)-modules.
Then inflation along the quotient map \(\scrr_{k\delta} \to \mathscr{S}_k\) gives the result.
\end{proof}

\begin{lem}\label{indprodS}
Let \(k,j \in \bbn\), and write \(n = k+j\). Then we have
\begin{align*}
\spe{((ke),(je))} \cong \widehat\scrm( \Delta^n(1^j) \otimes \Delta^n(1^k)) \langle j \rangle  \cong \widehat\scrm(\operatorname{L}^n(1^j) \otimes \operatorname{L}^n(1^k)) \langle j \rangle
\end{align*}
as \(\scrr_{n \delta}\)-modules.
\end{lem}
\begin{proof}
We have
\begin{align*}
\spe{((ke),(je))} &\cong
\spe{(je)} \circ \spe{(ke)} \langle j \rangle &\\
&\cong \widehat\scrm(\Delta(1^j)) \circ \widehat\scrm(\Delta(1^k)) \langle j \rangle & \textup{by \cref{Ske}}\\
&\cong \widehat\scrm_{n,j}(\Delta^n(1^j)) \circ \widehat\scrm_{n,k}(\Delta^n(1^k)) \langle j \rangle & \textup{by \cref{hirrev}}\\
&\cong \widehat\scrm(\Delta^n(1^j) \otimes \Delta^n(1^k)) \langle j \rangle & \textup{by \cref{MainKMimag}},
\end{align*}
where the first isomorphism follows by combining \cref{eq:duals}, \cite[Theorems 7.25 \& 8.2]{kmr}, \cite[Lemma 3.4.2]{km17}.
As \(\Delta(1^a) \cong \sL{1^a}\) for all \(a \in \bbn\), this completes the proof.
\end{proof}

Now, if \(V \in \scrr_{n \delta}\textup{-mod}\) is such that \(V \cong \widehat\scrm(W)\) for some \(W \in S_\bbf(h,n)\textup{-mod}\) (where we again consider \(\widehat\scrm(W)\) as an \(\scrr_{n \delta}\)-module via inflation), it necessarily implies that \((\textup{Ann}_{\scrr_{n \delta}}M_n)V = 0\).
Therefore 
\cref{Ske,indprodS} imply the immediate

\begin{cor}\label{lem:moritaequiv}
Let $n=j+k$.
We have \((\textup{Ann}_{\scrr_{n \delta}}M_n)\spe{((ke),(je))} = 0\), so that \(\spe{((ke),(je))}\) may be naturally considered as an \(\mathscr{S}_n\)-module.
Moreover,
\[
\scrm(\spe{((ke),(je))}) = \Delta^n(1^j) \otimes \Delta^n(1^k) \cong \Delta^n(1^k) \otimes \Delta^n(1^j) \in S_\bbf(n,n)\textup{-mod}.
\]
\end{cor}

\subsection{Indecomposable summands of imaginary Specht modules}\label{sec:indecsumm}
Next, it will also be useful to us to know a little about the indecomposable summands of Young permutation modules of the form $\sM{k,j}$ over the symmetric group in characteristic $p$.
These summands are by definition Young modules, and the information we need is the following result of Henke.

Suppose that $a$ and $b$ are positive integers with $p$-adic expansions
\[
a = a_0 + p a_1 + p^2 a_2 + \cdots \quad \text{ and } \quad b = b_0 + p b_1 + p^2 b_2 + \cdots
\]
with $0 \leq a_i, b_i < p$ for all $i$.
Write $a \leq_p b$ if, for all $i$, $a_i \leq b_i$.
It will sometimes be helpful to compactly write the $p$-adic expansions in the form $[a_0,a_1,a_2,\dots]$.

\begin{thmc}{henkepkostka}{Theorem 3.3}\label{thm:Henkeformula}
Let $\la = (n-j,j)$ and $\mu = (n-m,m)$.
Then the Young module $\sY{\mu}$ appears as a summand of $\sM{\la}$ exactly once if $j-m \leq_p n-2m$, and does not appear otherwise.
\end{thmc}

\begin{lem}\label{lem:no.indecompsummands}
The number of indecomposable summands of $\spe{((ke),(je))}$ is equal to that of the permutation module $\sM{k,j}$, which may be computed using \cref{thm:Henkeformula}.
\end{lem}

\begin{proof}
As in the proof of \cref{tiltdual}, the module $\Delta^{j+k}(1^k) \otimes \Delta^{j+k}(1^j) = \operatorname{L}^{j+k}(1^k) \otimes \operatorname{L}^{j+k}(1^j)$ may be naturally decomposed as a sum of tilting modules.
Applying the Schur functor to $\Delta^{j+k}(1^k) \otimes \Delta^{j+k}(1^j)$ yields the module $\sM{k,j}\otimes \sgn$, which is a signed version of the Young permutation module $\sM{k,j}$.
This decomposes as a direct sum of signed Young modules of the form $\sY{\la}\otimes \sgn$, which are the images of the tilting modules $\operatorname{T}(\la)$ under the Schur functor.
As tensoring with sign preserves indecomposable modules and direct sums, the result follows by application of \cref{indprodS}.
\end{proof}

\begin{prop}\label{prop:selfdualspecht}
For any $j,k\geq 1$, any $e\geq2$, and over any field $\bbf$,
\[
\spe{((ke),(je))}^\circledast \cong \spe{((ke),(je))} \langle -2j \rangle.
\]
Moreover, if $M$ is an indecomposable summand of $\spe{((ke),(je))}$, then $M^\circledast \cong M \langle -2j \rangle$.
\end{prop}
\begin{proof}
The indecomposable summands of \(\operatorname{L}^{j+k}(1^j) \otimes \operatorname{L}^{j+k}(1^k)\) are tilting modules, so the result follows immediately from \cref{indprodS} and \cref{tiltdual}.
\end{proof}

\section{Specht module homomorphisms}\label{sec:homs}

The main results of this section are the homomorphisms in \cref{prop:alphahom,prop:gammahom}.
The reader may want to skip ahead to the next section, where these homomorphisms are applied -- the computations carried out in order to prove that these maps are indeed homomorphisms are not so instructive.

In light of \cref{lem:switchcomps}, we can assume that $k\geq j$ as we determine the structure of $\spe{((ke),(je))}$.
In computing homomorphisms, we inherently make use of \cref{lem:tabresi} throughout this section -- it enables us to write the image of the cyclic generator $z_\la$ in terms of standard tableaux with residue $\bfi_\la$.

We recall the necessary notation and results from~\cite{ss20}, in preparation for computing Specht module homomorphisms with domain $\spe{((ke),(je))}$.

For $1\leq i\leq j\leq n-1$, we define $\s{j}{i}:= s_j s_{j-1}\dots s_i$ and $\su{i}{j}:=s_i s_{i+1}\dots s_j$.
Similarly, we define
\[
\psid x y := \psi_x \psi_{x-1} \dots \psi_{y} \quad \text{and} \quad \psiu y x := \psi_y \psi_{y+1} \dots \psi_x
\]
if $x \geq y$, and set both equal to $1_\bbf$ if $x<y$.

We now introduce notation for the basis vectors $v_{\ttt}$ of $\spe{((ke),(je))}$.
Observe that a standard $\la$-tableau $\ttt$ is determined by the entries $a_r:= \ttt(1,r,2)$ lying in its second component, for all $r \in \{1, \dots, je\}$.
We can thus write $\ttt = w_\ttt \ttt_{((ke),(je))}$, where
\[
w_\ttt := \s{a_1-1}{1} \s{a_2-1}{2} \dots \s{a_{je}-1}{je}\in\sss.
\]
It follows that $v_{\ttt} = \psi_\ttt z_{((ke),(je))}$ where
\[
\psi_\ttt := \psid{a_1-1}{1} \psid{a_2-1}{2} \dots \psid{a_{je}-1}{je} \in \scrr_n^\La.
\]
In order to distinguish our standard tableaux compactly, we will often write $v(a_1, a_2, \dots, a_{je})$ for the standard $\la$-tableau with entries $a_1, a_2, \dots, a_{je}$ in the second component.

We recall from \cite[\S4]{ss20} the notions of \emph{$e$-bricks}, \emph{$e$-brick tableaux} and \emph{brick transpositions}.
Let $\ttt$ be a standard $((ke),(je))$-tableau.
We define an \emph{$e$-brick} to be a sequence of $e$ adjacent nodes containing entries $je+1, je+2, \dots, (j+1)e$ for $j\geq 0$.
We say that $\ttt$ is an \emph{$e$-brick tableau} if all entries of $\ttt$ lie in $e$-bricks.
For any $e$-brick tableau $\ttt$, we number the $e$-bricks in the order of their entries, i.e.~$\ttt$ comprises of bricks $1, 2, \dots, j+k$.
Then we have brick transpositions and their corresponding $\psi$ expressions, which we will denote by $\Psi_r$.
In particular, the brick transposition that swaps the $r$th and $(r{+}1)$th bricks corresponds to
\[
\Psi_r = \psid{re}{(r-1)e+1} \chaind{re+1}{(r-1)e+2} \clapdots \chaind{(r+1)e-1}{re}.
\]
As with our $\psi$ generators, we introduce the shorthand $\Psid xy = \Psi_x \Psi_{x-1} \dots \Psi_y$ and $\Psiu yx = \Psi_y \Psi_{y+1} \dots \Psi_x$.

The following results will be crucial in our computations.

\begin{lemc}{ss20}{Lemma 6.8}\label{lem:psiaction}
	Let $e>2$, $\ttt \in \std\la$, $v_\ttt = v(a_1, \dots, a_{je})$, $1\leq r <n$, and $1\leq s < je$ such that $r\not\equiv 2s\pmod{e}$.
	\begin{enumerate}[label=(\roman*)]
		\item\label{lem:psiaction1} If $a_s=r$, $a_{s+1}=r+1$, then $\psi_r v(a_1, \dots, a_{je}) = 0$.
		
		\item\label{lem:psiaction2} If $s$ is maximal such that $a_s\leq{r-1}$, and $r, r+1 \notin \{a_1,\dots, a_{je}\}$, then $\psi_r v(a_1, \dots, a_{je}) = 0$.
	\end{enumerate}
\end{lemc}
%
%
%

\begin{lemc}{ss20}{Lemma 6.14(i) \& (ii)}\label{lem:newyaction}
	Let $e>2$, $0 \leq s \leq je-e$ and $v_\ttt = v(a_1, \dots, a_{je})$.
	\begin{enumerate}[label=(\roman*)]
		\item\label{lem:yaction2} 
		If $a_{s+e} = r$ for some $1\leq r \leq n$ such that $r \not \equiv 2s, 2s+1 \pmod{e}$ and $r-1,r+1, r+2, r+3, \dots, r+e-2 \not \in \{a_1, \dots, a_{je}\}$, then $y_{r-1} v_\ttt = 0$.
		
		\item\label{lem:yaction1}
		If $a_{s+e} = r$ for some $1\leq r \leq n$ such that $r \not \equiv 2s, 2s+1 \pmod{e}$ and $r+1, r+2, r+3, \dots, r+e-2 \not \in \{a_1, \dots, a_{je}\}$, then $y_r v_\ttt = 0$.
	\end{enumerate}
\end{lemc}


Analogues for the previous two lemmas hold for $e=2$, but were not proved in \cite{ss20}.
They can be proved by similar computations, and are needed to prove that \cref{prop:gammahom} also holds for $e=2$.

Next, we note properties of brick transpositions, KLR generators and basis vectors.


\begin{lemc}{ss20}{Lemma 4.11}
	\begin{enumerate}[label=(\roman*)]
		\item If $|r-s|>1$, then $\Psi_r \Psi_s = \Psi_s \Psi_r$.
		
		\item If $\la =((ke),(je))$, for some $j,k \geq1$ and $r\neq j$, then $\Psi_r z_\la = 0$.
	\end{enumerate}
\end{lemc}

\begin{propc}{ss20}{Lemmas 4.9 and 4.10 and Proposition 4.12}\label{prop:cancellation}
Let $\la = ((ke),(je))$, for some $j, k \geq 1$.
	If $v \in e(\bfi_\la) \spe\la$, and $1\leq r \leq j+k-1$, then
	\begin{enumerate}[label=(\roman*)]
		\item\label{prop:cancellation0a} $y_s v = 0$ for all $1\leq s \leq (k+j)e$;
		
		\item\label{prop:cancellation0b} $\psi_s v = 0$ for all $1\leq s \leq (k+j)e-1$ with $s\not\equiv 0 \pmod e$;
		
		\item\label{prop:cancellation1} $\psi_{re} \Psi_r v = -2 \psi_{re} v$;
		
		\item\label{prop:cancellation2} for $r<j+k-1$, $\psi_{re} \Psi_{r+1} \Psi_r v = \psi_{re} v$;
		
		\item\label{prop:cancellation3} for $r>1$, $\psi_{re} \Psi_{r-1} \Psi_r v = \psi_{re} v$.
	\end{enumerate}
\end{propc}

Now we are ready to compute our first homomorphisms.
The image of the cyclic generator under the homomorphism below is seen to be a linear combination of basis vectors indexed by brick-tableaux.

\begin{prop}\label{prop:alphahom}
Let $k\geq j\geq 1$.
There exists a degree 1 Specht module homomorphism
\begin{align*}
\alpha_{k,j}: \spe{((ke+e),(je-e))} &\longrightarrow \spe{((ke),(je))}\\
z_{((ke+e),(je-e))} &\longmapsto \sum_{i=0}^k (k+1-i) \Psid{j+i-1}{j} z_{((ke),(je))}.
\end{align*}
\end{prop}

\begin{proof}
Let $\la=((ke+e),(je-e))$ and $\mu = ((ke),(je))$.
We know from parts (i) and (ii) of \cref{prop:cancellation} that we need only show that $\psi_{re} \alpha_{k,j} (z_{\la}) = 0$ for each $r\in\{1,2,\dots,j-2\}\cup\{j,j+1,\dots,k+j-1\}$.
This is obvious if $r\in\{1,\dots,j-2\}$.
We now suppose that $r\in\{j,j+1,\dots,k+j-2\}$.

Firstly, if $s\geq r+2$ then
\[
\psi_{se} \Psid{r}{j} z_\mu
= \Psid{r}{j} \left( \psi_{se} z_\mu \right)
= 0.
\]

Secondly, if $s\leq r-2$ then
\[
\psi_{se} \Psid{r}{j} z_\mu
= \Psid{r}{s+2} \left( \psi_{se}\Psi_{s+1}\Psi_s \right) \Psid{s-1}{j} z_\mu
= \Psid{r}{s+2} \psi_{se} \Psid{s-1}{j} z_\mu
= 0,
\]
by \cref{prop:cancellation}.

We thus have
\begin{align*}
\psi_{re}\alpha_{k,j} (z_{\la})
&= \psi_{re} \left( (k-r+j-1) \Psid{r+1}{j} 
+ (k-r+j) \Psid{r}{j} + (k-r+j+1) \Psid{r-1}{j} \right) z_\mu\\
&= \left( (k-r+j-1) (\psi_{re}\Psi_{r+1}\Psi_r) \Psid{r-1}{j} 
+ (k-r+j) (\psi_{re}\Psi_r)\Psid{r-1}{j} + (k-r+j+1) \psi_{re} \Psid{r-1}{j} \right) z_{\mu}\\
&= \left((k-r+j-1) -2(k-r+j) +(k-r+j+1) \right) \Psid{r-1}{j} \psi_{re} z_\mu\\
&= 0,
\end{align*}
by \cref{prop:cancellation}.

Finally,
\begin{align*}
\psi_{(k+j-1)e}\alpha_{k,j} (z_{\la})
&= \psi_{(k+j-1)e} \left( \Psid{k+j-1}{j} + 2 \Psid{k+j-2}{j} \right) z_\mu\\
&= \left( (\psi_{(k+j-1)e}\Psi_{(k+j-1)})\Psid{k+j-2}{j} + 2 \psi_{(k+j-1)e} \Psid{k+j-2}{j} \right) z_{\mu}\\
&= 0.
\end{align*}
The degree of the homomorphism being 1 follows from an easy check that $\codeg \ttt = j$ for any $\ttt \in \std{(ke),(je)}$ with $\res\ttt = \bfi_{((ke),(je))}$.
\end{proof}

\begin{lem}\label{lem:psisinhom}
Let $k\geq j\geq 1$. Let $r=2je-e-2i$ for $i\geq 1$.
Then
\[
0=\psi_r v(1,\dots,\tfrac{r+e}{2},r+3,r+5,r+7,\dots,2je-e+1)\in\spe{((ke),(je))}.
\]
\end{lem}

\begin{proof}
We proceed by induction on $i$.
Suppose 
that $i=1$.
By applying part (i) of \cite[Corollary 6.9]{ss20}, we have
\begin{align*}
\psi_{2je-e-2} v(1,\dots,je-1,2je-e+1) 
& = \psi_{2je-e} \psi_{2je-e-2} v(1,\dots,je-1,2je-e)\\
& = \psi_{2je-e} v(1,\dots,je-1,2je-e-2)\\	
& = \psi_{2je-e} \psid{2je-e-3}{je} z_{\la}\\
& = \psid{2je-e-3}{je} \psi_{2je-e}z_{\la} \\
& =0.
\end{align*}
Now suppose that the statement holds for some $i>1$.
Then by applying part (i) of \cite[Corollary 6.9]{ss20}, we have
\begin{align*}
& \; \psi_r v(1,\dots,\tfrac{r+e}{2},r+3,r+5,r+7,\dots,2je-e+1) \\
= & \; \psi_{r+2} \psi_r v(1,\dots,\tfrac{r+e}{2},r+2,r+5,r+7,\dots,2je-e+1) \\
= & \; \psi_{r+2} v(1,\dots,\tfrac{r+e}{2},r,r+5,r+7,\dots,2je-e+1) \\
= & \; \psid{r-1}{\tfrac{r+e+2}{2}} \psi_{r+2} v(1,\dots,\tfrac{r+e+2}{2},r+5,r+7,\dots,2je-e+1)\\
= & \; 0 \text{ by the above inductive hypothesis.}\qedhere
\end{align*}
\end{proof}

We now come to our second family of homomorphisms.
Unlike $\alpha_{k,j}$ in \cref{prop:alphahom},
the image of the cyclic generator under a homomorphism $\gamma_{k,j}$ below is \emph{not} a linear combination of basis vectors indexed by brick-tableaux, as the generator has a different residue sequence.

\begin{prop}\label{prop:gammahom}
There is a degree $j$ Specht module homomorphism 
\begin{align*}
\gamma_{k,j}: \spe{((ke,je-e+1),(e-1))} &\longrightarrow \spe{((ke),(je))}\\
z_{((ke,je-e+1),(e-1))} & \longmapsto
v(1,2,\dots,e-1,e+1,e+3,e+5,\dots,2je-e+1).
\end{align*}
\end{prop}

\begin{proof}
We will prove this for $e\neq 2$, and omit the proof for $e=2$ -- it is similar enough in spirit, but rather lengthy.

Let $\la=((ke),(je))$.
We show that $v(1,2,\dots,e-1,e+1,e+3,e+5,\dots,2je-e+1)$ satisfies the relations that $z_{((ke,je-e+1),(e-1))}$ satisfies in \cref{defn:regularspecht}.

\begin{enumerate}[label=(\roman*)]
	
	\item We first show that all $y$ terms kill $v(1,2,\dots,e-1,e+1,e+3,e+5,\dots,2je-e+1) $.

\begin{itemize}

\item For $r\in\{1,\dots,e-1\}\cup\{2je-e+2,\dots,ke+je\}$, it is obvious that $y_r v(1,2,\dots,e-1,e+1,e+3,e+5,\dots,2je-e+1) = 0$.

\item

Suppose that $r\in\{e,e+2,e+4,\dots,2je-e\}$, and let $r=e+2i$ for $0 \leq i \leq je-e$.
Observe that
\[
y_rv(1,2,\dots,e-1,e+1,e+3,e+5,\dots,2je-e+1) 
= 
\epsilon \cdot
y_r
\psid{r}{e+i}
\psid{r+2}{e+i+1}
\psid{r+4}{e+i+2}
\dots
\psid{2je-e}{je}
z_{\la},
\]
where
\[
\epsilon:=\psi_e\psid{e+2}{e+1}\psid{e+4}{e+2}\dots\psid{r-2}{e+i-1}.
\]
To prove that the above expression is zero, we now proceed to show that
\begin{align*}
& y_{r+ae} \psid{r+ae}{(a+1)e+i}\psid{r+ae+2}{(a+1)e+i+1}\psid{r+ae+4}{(a+1)e+i+2}\dots\psid{2je-(a+1)e}{je} z_{\la} = 0
\end{align*}
for all $a\in\{0,\dots,j-1\}$. We proceed by reverse induction on $a$. For $a=j-1$, we must have $i=0$, in which case this expression becomes
\[
\left( y_{je}\psi_{je}(-1,0) \right) z_{\la}
= \psi_{je} y_{je+1} z_{\la}
= 0.
\]
Now assuming the above statement holds for some $a-1<j-1$, we have
\begin{align*}
& \left(y_{r+ae}\psi_{r+ae}(i-1,i)\right) \psid{r+ae-1}{(a+1)e+i}\psid{r+ae+2}{(a+1)e+i+1}\psid{r+ae+4}{(a+1)e+i+2}\dots\psid{2je-(a+1)e}{je} z_{\la}  \\
= \; &  \psid{r+ae}{(a+1)e+i}
\psi_{r+ae+2}
\left( y_{r+ae+1}\psi_{r+ae+1}(i,i) \right)
\psid{r+ae}{(a+1)e+i+1}\psid{r+ae+4}{(a+1)e+i+2}
\psid{r+ae+6}{(a+1)e+i+3}\dots\psid{2je-(a+1)e}{je} z_{\la}  \\
= \; &  \psid{r+ae}{(a+1)e+i}
\psi_{r+ae+2}
\left( \psi_{r+ae+1}y_{r+ae+2}-1 \right)
\psid{r+ae}{(a+1)e+i+1}\psid{r+ae+4}{(a+1)e+i+2}
\psid{r+ae+6}{(a+1)e+i+3}\dots\psid{2je-(a+1)e}{je} z_{\la}
\end{align*}

The first term is

\begin{align*}
& \psid{r+ae}{(a+1)e+i}
\psid{r+ae+2}{(a+1)e+i+1}
\psid{r+ae+4}{r+ae+3}
\left( y_{r+ae+2}\psi_{r+ae+2} (i+1,i) \right)
\psid{r+ae+1}{(a+1)e+i+2}\\
& \cdot \psid{r+ae+6}{(a+1)e+i+3}
\psid{r+ae+8}{(a+1)e+i+4}\dots\psid{2je-(a+1)e}{je} z_{\la}  \\
= \; & \psid{r+ae}{(a+1)e+i}
\psid{r+ae+2}{(a+1)e+i+1}
\psid{r+ae+4}{(a+1)e+i+2}
\psid{r+ae+6}{r+ae+4}
\left( y_{r+ae+3}\psi_{r+ae+3} (i+2,i) \right)
\psid{r+ae+2}{(a+1)e+i+3}\\
&\cdot \psid{r+ae+8}{(a+1)e+i+4}
\psid{r+ae+10}{(a+1)e+i+5}\dots\psid{2je-(a+1)e}{je} z_{\la}  \\
\vdots \;\; & \\
= \; &
\psid{r+ae}{(a+1)e+i} 
\psid{r+ae+2}{(a+1)e+i+1} 
\psid{r+ae+4}{(a+1)e+i+2} \dots
\psid{r+ae+2e-2}{(a+2)e+i-1}
y_{r+ae+e}\\
&\cdot \psid{r+ae+2e}{(a+2)e+i}
\psid{r+ae+2e+2}{(a+2)e+i+1}
\psid{r+ae+2e+4}{(a+2)e+i+2}
\dots
\psid{2je-(a+1)e}{je}z_{\la} \\
= \; &
\left(\psid{r+ae}{(a+1)e+i} 
\psid{r+ae+2}{(a+1)e+i+1} 
\dots
\psid{r+ae+2e-2}{(a+2)e+i-1} \right)
\left( \psid{r+ae+2e}{r+ae+e+1}
\psid{r+ae+2e+2}{r+ae+e+3}
\dots
\psid{2je-(a+1)e}{2je-(a+2)e+1} \right) \\
& \cdot 
y_{r+ae+e}
\psid{r+ae+e}{(a+2)e+i}
\psid{r+ae+e+2}{(a+2)e+i+1}
\psid{r+ae+e+4}{(a+2)e+i+2}
\dots
\psid{2je-(a+2)e}{je} z_{\la} \\
= \; & 0 \quad \text{by the inductive hypothesis.}
\end{align*}
The second term is 
\begin{align*}
-& \psi_{r+ae+2}
\left( \psi_{r+ae}\psi_{r+ae-1}\psi_{r+ae} (i-1,i,i-1) \right)
\psid{r+ae-2}{(a+1)e+i} \chaind{r+ae-1}{(a+1)e+i+1}\\
&\cdot \psid{r+ae+4}{(a+1)e+i+2} \psid{r+ae+6}{(a+1)e+i+3}
\dots \psid{2je-(a+1)e}{je} z_{\la} \\
= \; & 
\psi_{r+ae+2}
\left( \psi_{r+ae-1}\psi_{r+ae} \psi_{r+ae-1} + 1 \right)
\psid{r+ae-2}{(a+1)e+i} \chaind{r+ae-1}{(a+1)e+i+1}
\psid{r+ae+4}{(a+1)e+i+2} \psid{r+ae+6}{(a+1)e+i+3}
\dots \psid{2je-(a+1)e}{je} z_{\la}.
\end{align*}
The first term of this expression is zero by \cref{lem:psiaction}\ref{lem:psiaction1}, while the second term is zero by a computation analogous to \cref{lem:psisinhom}.


\item Suppose that $r\in\{e+1,e+3,\dots,2je-e+1\}$, and let $r=e+2i+1$ for $0 \leq i \leq je-e$.
Observe that
\[
y_rv(1,2,\dots,e-1,e+1,e+3,e+5,\dots,2je-e+1) 
= 
\epsilon \cdot
y_r
\psid{r-1}{i+e}\psid{r+1}{i+e+1}\psid{r+3}{i+e+2}\dots\psid{2je-e}{je}z_{\la},
\]
where
\[
\epsilon:=\psi_e\psid{e+2}{e+1}\psid{e+4}{e+2}\dots\psid{r-3}{i+e-1}.
\]
It thus suffices to show that $y_r \psid{r-1}{i+e}\psid{r+1}{i+e+1}\psid{r+3}{i+e+2}\dots\psid{2je-e}{je}z_{\la}=0$.
We have
\begin{align*}
 y_r \psid{r-1}{i+e}\psid{r+1}{i+e+1}\psid{r+3}{i+e+2}\dots\psid{2je-e}{je}z_{\la}	
&=
\left( y_r \psi_{r-1} (i-1,i) \right)
\psid{r-2}{i+e}\psid{r+1}{i+e+1}\psid{r+3}{i+e+2}\dots\psid{2je-e}{je}z_{\la}\\
&=
\psi_{r-1} 
 y_{r-1}
\psid{r-2}{i+e}\psid{r+1}{i+e+1}\psid{r+3}{i+e+2}\dots\psid{2je-e}{je}z_{\la}.
\end{align*}
If $i=0$, then the above is
\begin{align*}
\psi_e y_e \psid{e+2}{e+1}\psid{e+4}{e+2}\dots\psid{2je-e}{je}z_{\la}
= \psi_e\psid{e+2}{e+1}\psid{e+4}{e+2}\dots\psid{2je-e}{je} y_e z_{\la}
=0.
\end{align*}
If $i>0$, then the above becomes
\begin{align*}
&
\psi_{r-1} 
\left( y_{r-1} \psi_{r-2} (i-1,i-1) \right)
\psid{r-3}{i+e}\psid{r+1}{i+e+1}\psid{r+3}{i+e+2}\dots\psid{2je-e}{je}z_{\la}\\
=  \;
&\psi_{r-1} 
\left( \psi_{r-2} y_{r-2} + 1 \right)
\psid{r-3}{i+e}\psid{r+1}{i+e+1}\psid{r+3}{i+e+2}\dots\psid{2je-e}{je}z_{\la}.
\end{align*}
The second term of this expression becomes
\begin{align*}
\psi_{r-1} 
\psid{r-3}{i+e}\psid{r+1}{i+e+1}\psid{r+3}{i+e+2}\dots\psid{2je-e}{je}z_{\la} 
&= 
\psid{r-3}{i+e}
\psi_{r-1}
v(1,\dots,i+e,r+2,r+4,\dots,2je-e+1)\\
&= 0 \text{ by \cref{lem:psisinhom}.}
\end{align*}
If $i=1$, then the first term becomes
\begin{align*}
& \psid{e+2}{e+1} y_{e+1} 
\psid{e+4}{e+2}\psid{e+5}{e+3}\dots\psid{2je-e}{je}z_{\la}	
= \psid{e+2}{e+1}\psid{e+4}{e+2}\psid{e+5}{e+3}\dots\psid{2je-e}{je}
y_{e+1} z_{\la}
=0.
\end{align*}
If $i>1$, then this term becomes
\begin{align*}
\psid{r-1}{r-2} y_{r-2}
v(1,\dots,i+e-1,r-2,r+2,r+4,\dots,2je-e+1) = 0,
\end{align*}
by \cref{lem:newyaction}\ref{lem:yaction1}. 
If $e>5$, we have
\begin{align*}
& \psi_{r-1}\psi_{r-2} 
\left( y_{r-2} \psi_{r-3} (i-1,i-2) \right)
\psid{r-4}{i+e}\psid{r+1}{i+e+1}\psid{r+3}{i+e+2}\dots\psid{2je-e}{je}z_{\la} \\
=  \;
&\psi_{r-1}\psi_{r-2} \psi_{r-3}
y_{r-3} 
\psid{r-4}{i+e}\psid{r+1}{i+e+1}\psid{r+3}{i+e+2}\dots\psid{2je-e}{je}z_{\la}.
\end{align*}
If $i=2$, then this expression becomes
\begin{align*}
\psid{e+4}{e+2}
y_{e+2} 
\psid{e+6}{e+3}\psid{e+8}{e+4}\dots\psid{2je-e}{je}z_{\la}
=
\psid{e+4}{e+2}
\psid{e+6}{e+3}\psid{e+8}{e+4}\dots\psid{2je-e}{je}y_{e+2} z_{\la}
=0.
\end{align*}
Now assuming that $i>2$, the expression is
\[
\psi_{r-1}\psi_{r-2} \psi_{r-3}
y_{r-3} 
v(1,\dots,i+e-1,r-3,r+2,r+4,\dots,2je-e+1) = 0
\]
by \cref{lem:newyaction}\ref{lem:yaction1} if $e=6$.
If $e>6$, we continue in this way until we reach
\begin{align*}
& \psid{r-1}{2i+4}
y_{2i+4}
v(1,\dots,i+e-1,2i+4,r+2,r+4,\dots,2je-e+1) = 0
\end{align*}
by \cref{lem:newyaction}\ref{lem:yaction1}.
\end{itemize}

\item We now show that $v(1,\dots,e-1,e+1,e+3,\dots,2je-e+1)$ satisfies each of the $\psi$ relations in \cref{defn:regularspecht}.

\begin{itemize}
	
\item For all $r\in\{1,\dots,e-2\}\cup\{2je-e+2,\dots,ke+je-1\}$, it is obvious that $\psi_r$ kills $v(1,\dots,e-1,e+1,e+3,\dots,2je-e+1)$.
	
\item Suppose that $r\in\{e,e+2,e+4,\dots,2je-e\}$, and let $r=e+2i$ for $0 \leq i \leq je-e$.
Then
	
\begin{align*}
& \psi_r\psi_{r+1}
v(1,\dots,e-1,e+1,e+3,\dots,2je-e+1) \\
= \; &
\psi_e\psid{e+2}{e+1}\psid{e+4}{e+2}\dots\psid{r-2}{e+i-1}
\psi_r
v(1,\dots,e+i-1,r+2,r+3,r+5,r+7,\dots,2je-e+1)\\
= \; & 0 \text{ by \cref{lem:psiaction}\ref{lem:psiaction2}}.
\end{align*}
\item Suppose that $r\in\{e+2,e+4,\dots,2je-e\}$, and let $r=e+2i$ for $1\leq i \leq je-e$.
Then
\begin{align*}
& \psi_r\psi_{r-1}
v(1,\dots,e-1,e+1,e+3,\dots,2je-e+1) \\
= \; &
\psi_e\psid{e+2}{e+1}\psid{e+4}{e+2}\dots\psid{r-4}{e+i-2}
\psi_r
v(1,\dots,e+i-2,r,r+1,r+3,r+5,\dots,2je-e+1)\\
= \; & 0 \text{ by \cref{lem:psiaction}\ref{lem:psiaction1}.}	
\end{align*}

\item Suppose that $r\in\{e,e+2,e+4,\dots,2je-e\}$, and let $r=e+2i$ for some $0 \leq i \leq je-e$.
Observe that
\begin{align*}
& \psi_r v(1,\dots,e-1,e+1,e+3,\dots,2je-e+1) \\
= \; &	
\psi_e\psid{e+2}{e+1}\psid{e+4}{e+2}\dots\psid{r-2}{e+i-1}
\left( \psi_r^2 e(i-1, i) \right) \psid{r-1}{e+i}\psid{r+2}{e+i+1}\psid{r+4}{e+i+2}\dots\psid{2je-e}{je}z_{\la}\\
%
%
= \; &
\psi_e\psid{e+2}{e+1}\psid{e+4}{e+2}\dots\psid{r-2}{e+i-1}
\left( y_r -y_{r+1} \right) \psid{r-1}{e+i}\psid{r+2}{e+i+1}\psid{r+4}{e+i+2}\dots\psid{2je-e}{je}z_{\la}.
\end{align*}
In the first term, we have
\begin{align*}
& \psi_e\psid{e+2}{e+1}\psid{e+4}{e+2}\dots\psid{r-2}{e+i-1}
\left( y_r \psi_{r-1} e(i-1, i-1) \right) \psid{r-2}{e+i}\psid{r+2}{e+i+1}\psid{r+4}{e+i+2}\dots\psid{2je-e}{je}z_{\la} \\
= \; &
\psi_e\psid{e+2}{e+1}\psid{e+4}{e+2}\dots\psid{r-2}{e+i-1}
\left( \psi_{r-1} y_{r-1} + 1 \right) \psid{r-2}{e+i}\psid{r+2}{e+i+1}\psid{r+4}{e+i+2}\dots\psid{2je-e}{je}z_{\la},
\end{align*}
and the first term can be shown to be $0$, as at the end of part (i) of this proof.
We now let
\[
\zeta:= \psid{r+2}{e+i+1}\psid{r+4}{e+i+2}\dots\psid{2je-e}{je}.
\]
By applying \cref{lem:psiaction}\ref{lem:psiaction1} throughout, we thus have
\begin{align*}
& \psi_e\psid{e+2}{e+1}\psid{e+4}{e+2}\dots\psid{r-4}{e+i-2}
\left( \psi_{r-2}\psi_{r-3}\psi_{r-2} e(i-2,i-1,i-2) \right)
\psid{r-4}{e+i-1} \psid{r-3}{e+i}\cdot\zeta z_{\la} \\
= \; &
\psi_e\psid{e+2}{e+1}\psid{e+4}{e+2}\dots\psid{r-4}{e+i-2}
\left( \cancel{\psi_{r-3}\psi_{r-2}\psi_{r-3}} +1 \right)
\psid{r-4}{e+i-1} \psid{r-3}{e+i}\cdot\zeta z_{\la} \\
= \; &
\psi_e\psid{e+2}{e+1}\psid{e+4}{e+2}\dots\psid{r-6}{e+i-3}
\left( \psi_{r-4}\psi_{r-5}\psi_{r-4} e(i-3,i-2,i-3) \right)
\psid{r-6}{e+i-2} \psid{r-5}{e+i-1}\psid{r-3}{e+i}\cdot\zeta z_{\la}\\
= \; &
\psi_e\psid{e+2}{e+1}\psid{e+4}{e+2}\dots\psid{r-6}{e+i-3}
\left( \cancel{\psi_{r-5}\psi_{r-4}\psi_{r-5}} +1 \right)
\psid{r-6}{e+i-2} \psid{r-5}{e+i-1}\psid{r-3}{e+i}\cdot\zeta z_{\la}\\
\vdots \; \; & \\
= \; &
\psi_e \psid{e+2}{e+1}
\left( \psi_{e+4}\psi_{e+3}\psi_{e+4} e(1,2,1) \right) \psi_{e+2} \psi_{e+3}\psid{e+5}{e+4}\psid{e+7}{e+5}\dots\psid{r-3}{e+i} \cdot\zeta z_{\la}\\
= \; &
\psi_e \psid{e+2}{e+1}
\left( \cancel{\psi_{e+3}\psi_{e+4}\psi_{e+3}} +1 \right) \psi_{e+2} \psi_{e+3}\psid{e+5}{e+4}\psid{e+7}{e+5}\dots\psid{r-3}{e+i} \cdot\zeta z_{\la}\\
= \; &
\psi_e \left( \psi_{e+2}\psi_{e+1}\psi_{e+2} e(0,1,0) \right)
\psi_{e+3} \psid{e+5}{e+4}\psid{e+7}{e+5}\dots\psid{r-3}{e+i} \cdot\zeta z_{\la}\\
= \; &
\psi_e \left( \psi_{e+1}\psi_{e+2}\psi_{e+1} +1 \right)
\psi_{e+3} \psid{e+5}{e+4}\psid{e+7}{e+5}\dots\psid{r-3}{e+i} \cdot\zeta z_{\la}\\
= \; &
\psi_e \psi_{e+1}\psi_{e+2}
\psi_{e+3} \psid{e+5}{e+4}\psid{e+7}{e+5}\dots\psid{r-3}{e+i} \cdot\zeta \left(\psi_{e+1} z_{\la}\right)
+ 
\psi_{e+3} \psid{e+5}{e+4}\psid{e+7}{e+5}\dots\psid{r-3}{e+i} \cdot\zeta \left( \psi_e z_{\la} \right)\\
= \; & 0.
\end{align*}
Now let
\[
\eta:= \psi_e\psid{e+2}{e+1}\psid{e+4}{e+2}\dots\psid{r-2}{e+i-1}.
\]
Then the second $y$ term of the expression above becomes
\begin{align*}
& \eta\cdot
\psid{r-1}{e+i}
\psi_{r+2} \left( y_{r+1} \psi_{r+1} (i,i) \right) \psid{r}{e+i+1}
\psid{r+4}{e+i+2}\psid{r+6}{e+i+3}\dots\psid{2je-e}{je} z_{\la} \\
= \; & 
\eta\cdot
\psid{r-1}{e+i}
\psi_{r+2} \left( \psi_{r+1} y_{r+2} -1 \right) \psid{r}{e+i+1}
\psid{r+4}{e+i+2}\psid{r+6}{e+i+3}\dots\psid{2je-e}{je} z_{\la},
\end{align*}
where $\psi_{r+2} \psid{r+4}{e+i+2}\psid{r+6}{e+i+3}\dots\psid{2je-e}{je} z_{\la} = 0$ by \cref{lem:psisinhom}.
We thus have
\begin{align*}
& \eta\cdot
\psid{r-1}{e+i}\psid{r+2}{e+i+1} y_{r+2}
\psid{r+4}{e+i+2}\psid{r+6}{e+i+3}\dots\psid{2je-e}{je} z_{\la} \\
= \; &
\eta\cdot
\psid{r-1}{e+i}\psid{r+2}{e+i+1}
\psid{r+4}{r+3}\psid{r+6}{r+5}\dots\psid{2je-e}{2je-e-1}
y_{r+2} \psid{r+2}{e+i+2}\psid{r+4}{e+i+3}\dots\psid{2je-e-2}{je} z_{\la},
\end{align*}
where $y_{r+2} \psid{r+2}{e+i+2}\psid{r+4}{e+i+3}\dots\psid{2je-e-2}{je} z_{\la} = 0$ as in the proof of the $y_r$ relations above.
\end{itemize}
\end{enumerate}
%
%
%
Finally, we prove that $\gamma_{k,j}$ has degree $j$.
If $e=2$, then $\ttt_{((ke,je-e+1),(e-1))}$ has degree $2j$ and the degree of the tableau corresponding to $v(1,2,\dots,e-1,e+1,e+3,e+5,\dots,2je-e+1)$ is $3j$.
If $e>2$, the above degrees are $1$ and $j+1$, respectively.
\end{proof}

\section{Main results}\label{sec:main}

To begin with, we introduce some notation.
We will write
\[
M \cong L_1 \hspace{2pt} |\hspace{1pt} L_2 \hspace{2pt} |\hspace{1pt} \dots \hspace{2pt} |\hspace{1pt} L_r
\]
to mean that the module $M$ is uniserial with composition factors $L_1, \dots, L_r$ listed from socle to head.
We will write
\[
M \cong L_1 \hspace{2pt} |\hspace{1pt} \dots \hspace{2pt} |\hspace{1pt} L_r \oplus \dots \oplus N_1 \hspace{2pt} |\hspace{1pt} \dots \hspace{2pt} |\hspace{1pt} N_r
\]
to mean that the module is isomorphic to the direct sum of uniserial modules $L_1 \hspace{2pt} |\hspace{1pt} \dots \hspace{2pt} |\hspace{1pt} L_r, \dots, N_1 \hspace{2pt} |\hspace{1pt} \dots \hspace{2pt} |\hspace{1pt} N_r$. We will also have occasion, in \cref{prop:j=2}(vi) and \cref{eg:5factorsummand}, to indicate module structure via Alperin diagrams, see \cite{alperin}.

We will assume that $k,j\geq 1$ throughout, otherwise the Specht module $\spe{((ke),(je))}$ is simple.
First, we handle the case where $j=1$.

\begin{prop}\label{prop:j=1}
Suppose $k\geq 1$, and let $n = ke + e$.
Then if $p\nmid k+1$, $\spe{((ke),(e))}$ is semisimple and is isomorphic to $\D{\left( \left(ke,1\right),(e-1) \right)}\langle 1\rangle \oplus \D{((n),\varnothing)}\langle 1\rangle$.
If $p \mid k+1$, then
\[
\spe{((ke),(e))} \cong \D{((n),\varnothing)}\langle 1\rangle \hspace{2pt} |\hspace{1pt} \D{\left( \left(ke,1\right),(e-1) \right)}\langle 1\rangle \hspace{2pt} |\hspace{1pt} \D{((n),\varnothing)}\langle 1\rangle.
\]
\end{prop}

\begin{proof}
Using the Morita equivalence $\scrm$ of \cref{sec:imag}, we have that $\scrm(\spe{((ke),(e))}) \cong \Delta(1^k) \otimes \Delta(1)$, which has a filtration by the modules $\Delta(1^{k+1}) = \sL{1^{k+1}}$ and $\Delta(2,1^{k-1})$, by \cref{lem:Pierifilt}.
By \cref{cor:irredweyls}, $\Delta(2,1^{k-1})$ is irreducible if and only if $p\nmid k+1$.
Since taking tensor products commutes with duality, and all simple modules of Schur algebras are self-dual, $\Delta(1^k) \otimes \Delta(1) = \sL{1^k}\otimes \sL1$ must be self-dual.
It follows that if $p \nmid k+1$, $\Delta(1^k) \otimes \Delta(1) \cong \sL{1^{k+1}} \oplus \sL{2,1^{k-1}}$.
This implies that $\scrm(\spe{((ke),(e))})$ must also be a direct sum of two simple modules, since it is the Morita pre-image of $\Delta(1^k) \otimes \Delta(1)$.
Finally, by \cref{prop:alphahom,prop:gammahom}, we know that $\D{\left( \left(ke,1\right),(e-1) \right)}\langle 1\rangle$ and $\D{((n),\varnothing)}\langle 1\rangle$ are composition factors of $\scrm(\spe{((ke),(e))})$, which completes the proof if $p\nmid k+1$.

If $p\mid k+1$, then $\Delta(2,1^{k-1}) \cong \sL{1^{k+1}} \hspace{2pt} |\hspace{1pt} \sL{2,1^{k-1}}$. 
Since $\Delta(1^k) \otimes \Delta(1)$ must be self-dual, it follows that $\Delta(1^k) \otimes \Delta(1) \cong \sL{1^{k+1}} \hspace{2pt} |\hspace{1pt} \sL{2,1^{k-1}} \hspace{2pt} |\hspace{1pt} \sL{1^{k+1}}$.
\cref{prop:alphahom} tells us that $\D{((n),\varnothing)}\langle 1\rangle$ is a submodule of $\spe{((ke),(e))}$, while \cref{prop:gammahom} implies that $\D{\left( \left(ke,1\right),(e-1) \right)}\langle 1\rangle$ is a composition factor of $\spe{((ke),(e))}$.
It follows that $\spe{((ke),(e))} \cong \D{((n),\varnothing)}\langle 1\rangle \hspace{2pt} |\hspace{1pt} \D{\left( \left(ke,1\right),(e-1) \right)}\langle 1\rangle \hspace{2pt} |\hspace{1pt} \D{((n),\varnothing)}\langle r \rangle$ for some $r\in\bbz$.
To see that $r=1$ as claimed, we apply \cref{prop:selfdualspecht}.
\end{proof}

\subsection{The semisimple cases}\label{subsec:ss}

In this section, we will handle all the semisimple Specht modules and determine their decompositions.

\begin{thm}\label{thm:k>jss}
Suppose $k \geq j\geq 1$.
Then $\spe{((ke),(je))}$ is semisimple if and only if one of the following holds:
\begin{itemize}
\item $p\neq 2$ and $p$ does not divide any of the integers $k+j, k+j-1, \dots, k-j+2$;

\item $p=2$, $j=1$, and $k$ is even;

\item $p=2$, $j=2$, and $k\equiv 1 \pmod 4$.
\end{itemize}
When semisimple, $\spe{((ke),(je))}$ is isomorphic to
\begin{align*}
&\D{(((k+j-1)e,1), (e-1))} \langle j\rangle
\oplus \D{(((k+j-2)e,e+1), (e-1))} \langle j\rangle
\oplus \cdots \oplus \D{((ke,(j-1)e+1),(e-1))} \langle j\rangle
\oplus \D{((ke+je),\varnothing)}\langle j\rangle\\
= \; &\bigoplus_{r=1}^{j} \D{\left( \left((k+j-r)e,(r-1)e+1\right),(e-1) \right)}\langle j\rangle \oplus \D{((ke+je),\varnothing)}\langle j\rangle.
\end{align*}
\end{thm}

\begin{proof}
For $j=1$, the result is contained in \cref{prop:j=1}.
So assume now that $j>1$.

Note that under the Morita equivalence $\scrm$, our Specht module is mapped to a module that is filtered by the Weyl modules $\Delta{(1^{k+j})}, \Delta{(2,1^{k+j-2})}, \dots, \Delta{(2^j,1^{k-j})}$, by \cref{lem:Pierifilt}.
Since Weyl modules are \emph{always} indecomposable, $\spe{((ke),(je))}$ cannot possibly be semisimple unless each of those Weyl modules is irreducible.
By \cref{cor:irredweyls}, this happens exactly in the cases of our theorem statement, which proves the `only if' part of the theorem.

For the `if' part, it suffices to note that under the stated conditions, $\scrm(\spe{((ke),(je))})$ is a self-dual module with simple factors $\sL{1^{k+j}}, \sL{2,1^{k+j-2}}, \dots, \sL{2^j,1^{k-j}}$, each occurring exactly once.

Finally, we prove the stated decomposition by induction on $j$, keeping $k+j$ fixed.
For $j=1$, the result is contained in \cref{prop:j=1}, so we may assume that
\[
\uspe{((ke+e),(je-e))} \cong \bigoplus_{r=1}^{j-1} \uD{\left( \left((k+j-r)e,(r-1)e+1\right),(e-1) \right)} \oplus \uD{((ke+je),\varnothing)}.
\]
We also know that
\[
\scrm(\spe{((ke),(je))}) \cong \sL{1^{k+j}} \oplus \sL{2,1^{k+j-2}} \oplus \dots \oplus \sL{2^{j-1},1^{k-j+2}} \oplus \sL{2^j,1^{k-j}},
\]
while 
\[
\scrm(\spe{((ke+e),(je-e))}) \cong \sL{1^{k+j}} \oplus \sL{2,1^{k+j-2}} \oplus \dots \oplus \sL{2^{j-1},1^{k-j+2}}.
\]
It follows that $\scrm(\D{((ke+je),\varnothing)}) \cong \sL{1^{k+j}}$ and $\scrm(\D{\left( \left((k+j-r)e,(r-1)e+1\right),(e-1) \right)}) \cong \sL{2^r,1^{k+j-2r}}$ for $r=1, \dots, j-1$.
To see that $\scrm(\D{\left( \left(ke,je-e+1\right),(e-1) \right)}) \cong \sL{2^j,1^{k-j}}$, it suffices to note that $\spe{((ke+e),(je-e))} \subseteq \spe{((ke),(je))}$, and that these two Specht modules only differ by a single extra simple summand whose image under $\scrm$ is $\sL{2^j,1^{k-j}}$.
\cref{prop:gammahom} tells us that $\D{\left( \left(ke,je-e+1\right),(e-1) \right)}\langle j \rangle$ is a composition factor of $\spe{((ke),(je))}$.
Since this simple is not already accounted for, the result follows.
To see the grading shifts of the other simple factors, we may argue the same as in the proof of \cref{prop:j=1}.
\end{proof}

\begin{eg}
For $e=3$, $k=7$, and $j=5$,
\[
\spe{((21),(15))} \langle -5 \rangle \cong \D{((36),\varnothing)} \oplus \D{((33,1),(2))} 
\oplus \D{((30,4),(2))} 
\oplus \D{((27,7),(2))} 
\oplus \D{((24,10),(2))}\\ 
\oplus \D{((21,13),(2))} .
\]
\end{eg}

It will be useful to introduce some notation for the combinatorial map on labels matching an $\sL\la$ with $\D\mu$ that is implicit in the above proof.

\begin{defn}\label{def:moritamaponlabels}
We let $\scrt$ denote the combinatorial map that sends a two-column partition $\mu$ to a bipartition $\la$ precisely when $\scrm(\D\la) = \sL\mu$.
In particular, given a two-column partition $\mu$ of $n$, we define
\begin{align*}
\scrt(\mu):=
\begin{cases}
((ne),\varnothing)
&\text{ if $\mu=(1^n)$;}\\
(((n-m)e,(m-1)e+1),(e-1))
&\text{ if $\mu = (2^m,1^{n-2m}) $ with $m\geqslant 1$.  }
\end{cases}
\end{align*}
\end{defn}

\begin{eg}
For any $e$, $\scrt(1^7) = ((7e),\varnothing)$, $\scrt(2,1^5) = ((6e,1),(e-1))$, $\scrt(2^2,1^3) = ((5e,e+1),(e-1))$, and $\scrt(2^3,1) = ((4e,2e+1),(e-1))$.
\end{eg}

From \cref{thm:k>jss}, applying $i$-induction functors yields:

\begin{thm}\label{thm:generalised k>j ss}
Let $\la = ((ke + a, 1^b), (je + a, 1^b))$, for some $k \geq j\geq 1$, with $0 < a \leq e$ and $0 \leq b < e$ with $a + b \neq e$, or for $a = b = 0$ and set $n:=|\la|$.
Suppose that $p$ does not divide any of the integers $k+j, k+j-1, \dots, k-j+2$, or else that $p = 2 = j$ and $k\equiv 1 \pmod 4$.
Then,
\[
\spe\la \cong \begin{cases}
\bigoplus_{r=1}^{j} \D{\left( \left(n-re-a,(r-1)e+1+a\right),(e-1) \right)}\langle j \rangle \oplus \D{((n-a), (a))} \langle j \rangle & \text{if } b = 0,\\
\\
\bigoplus_{r=1}^{j} \D{\left( \left(n-re-a-2b,(r-1)e+1+a, 1^{b-1}\right),(e,1^b) \right)}\langle j \rangle \oplus \D{((n-a-2b,1^b),(a,1^b))} \langle j \rangle & \text{if } b \geq 1,
\end{cases}
\]
\end{thm}

\begin{proof}
The proof is in essence a similar argument to that of~\cite[Propositions 4.3 and 4.4]{ss20}.
As there, we apply the divided power functors of \cref{subsec:divpowers} to $\spe{((ke),(je))}$, whose decomposition we know by \cref{thm:k>jss}.
In particular, if $a+b<e$, we examine
\[
f_{e-b}^{(2)} f_{e-b+1}^{(2)} \dots f_{e-1}^{(2)} \cdot f_{a-1}^{(2)} f_{a-2}^{(2)} \cdots f_{0}^{(2)} \spe{((ke),(je))},
\]
while if $a+b>e$, we examine
\[
f_{e-b}^{(2)} f_{e-b+1}^{(2)} \cdots f_{a-3}^{(2)} f_{a-2}^{(2)} \cdot f_{a-1}^{(4)} \cdot f_{a}^{(2)} f_{a+1}^{(2)} \dots f_{e-1}^{(2)} \cdot f_{a-2}^{(2)} f_{a-3}^{(2)} \cdots f_{0}^{(2)} \spe{((ke),(je))}.
\]
By the cyclotomic analogues of the arguments used in the proof of \cite[Theorem~3.2]{ls14}, if $\phi_i = \phi_i(\spe\mu)$ is the number of addable $i$-nodes of $\mu$, then $f_i^{(\phi_i)} \spe\mu$ is the Specht module labelled by the bipartition obtained by adding all of these addable $i$-nodes to $\mu$.
At each step, we apply a divided power $f_i^{(2)}$ to a Specht module $\spe\mu$, where $\mu$ has exactly two addable $i$-nodes and no removable $i$-nodes, and thus $f_i^{(2)}\spe\mu = \spe\nu$, where $\nu$ is obtained from $\mu$ by adding both of these $i$-nodes.
The only exception is that if $a+b>e$, there is one step at which $\mu$ has exactly four addable $i$-nodes, and $f_i^{(4)}\spe\mu = \spe\nu$, where $\nu$ is obtained from $\mu$ by adding all four of these $i$-nodes.
So each case above yields the Specht module $\spe\la$.

Now, \cref{thm:k>jss} gives the decomposition of $\spe{((ke),(je))}$, and so it suffices to check the result of applying our series of divided power functors (which are exact) to the simple modules appearing in the decomposition.
Indeed, we will show that for each simple summand and at each step, we have $f_i^{(2)}\D\mu = \D\nu$ for some $\nu$, eventually leading to the desired decomposition.

First, we begin with the case $b=0$, and examine the summand $\D{((ke+je),\varnothing)}\langle j \rangle$.
At each step in $f_{a-1}^{(2)} f_{a-2}^{(2)} \cdots f_{0}^{(2)} \D{((ke+je),\varnothing)}$, we have some $f_i^{(2)} \D\mu$, where $\mu$ has exactly two addable $i$-nodes (one at the end of the first row of each component), and no removable $i$-nodes.
Thus both addable $i$-nodes are conormal, and it follows from \cref{lem:grbranch} that $f_i^2 \D\mu = \D{\tilde{f_i}^2 \mu} \langle -1 \rangle \oplus \D{\tilde{f_i}^2 \mu} \langle 1 \rangle$, and thus $f_i^{(2)} \D\mu =  \D{\tilde{f_i}^2 \mu}$.

A similar argument applies to the other summands, noting that if $i < e-2$, then the bipartition $\left( \left((j+k-r)e+i,(r-1)e+1+i\right),(e-1) \right)$ has exactly two addable $i$-nodes -- in the first two rows of the first component -- and no removable $i$-nodes.
If $i=e-2$, then the above bipartition has an addable $i$-node in each of the first three rows of the first component, and a removable $i$-node at the end of the first row of the second component.
In both cases, there are exactly two conormal $i$-nodes.

The case $b\neq 0$ is almost identical, with conormal nodes now appearing in the first column of each component in many of the steps.
The only non-trivial difference is when $a+b > e$, and we are applying the functor $f_{a-1}^{(4)}$ to a simple module $f_{a}^{(2)} f_{a+1}^{(2)} \dots f_{e-1}^{(2)} \cdot f_{a-2}^{(2)} f_{a-3}^{(2)} \cdots f_{0}^{(2)} \D\mu$.
In each case, we are applying $f_{a-1}^{(4)}$ to a module $\D\nu = f_{a}^{(2)} f_{a+1}^{(2)} \dots f_{e-1}^{(2)} \cdot f_{a-2}^{(2)} f_{a-3}^{(2)} \cdots f_{0}^{(2)} \D\mu$ indexed by a bipartition $\nu$ that has exactly four addable $(a{-}1)$-nodes and no removable $(a{-}1)$-nodes; in particular, $\nu$ has exactly four conormal $(a{-}1)$-nodes.
Applying the functor $f_{a-1}$ to the head four times tells us that the desired module has head $\D{\tilde{f}_{a-1}^{(4)}\nu}\langle -6 \rangle$ and a similar application to the socle tells us that the desired module has socle $\D{\tilde{f}_{a-1}^{(4)}\nu}\langle 6 \rangle$, both by \cref{lem:grbranch}.
Since $f_i^{4} \cong [4]! f_i^{(4)}$, it follows that $f_{a-1}^{(4)} \D\nu = \D{\tilde{f}_{a-1}^4 \nu}$.
Finally, the degree shifts are unchanged by application of these functors, since $f_i^{(r)} (\D\mu \langle k \rangle) = f_i^{(r)} (\D\mu) \langle k \rangle$, as \(f_i^{(r)}\) is a graded functor.
\end{proof}

It will be useful to introduce some notation for the combinatorial map on simple labels taking labels of composition factors of $\spe{((ke),(je))}$ to the corresponding labels of composition factors of $\spe\la$ as in the above proof.

\begin{defn}\label{def:bipartitioninduction}
We let $\calf_{a,b}$ denote the map on bipartitions defined by
\[
\calf_{a,b}(\mu) = \begin{cases}
\tilde{f}_{e-b}^{2} \tilde{f}_{e-b+1}^{2} \dots \tilde{f}_{e-1}^{2} \cdot \tilde{f}_{a-1}^{2} \tilde{f}_{a-2}^{2} \cdots \tilde{f}_{0}^{2} \mu & \text{ if $a+b<e$,}\\
\\
\tilde{f}_{e-b}^{2} \tilde{f}_{e-b+1}^{2} \cdots \tilde{f}_{a-3}^{2} \tilde{f}_{a-2}^{2} \cdot \tilde{f}_{a-1}^{4} \cdot \tilde{f}_{a}^{2} \tilde{f}_{a+1}^{2} \dots \tilde{f}_{e-1}^{2} \cdot \tilde{f}_{a-2}^{2} \tilde{f}_{a-3}^{2} \cdots \tilde{f}_{0}^{2} \mu & \text{ if $a+b>e$.}
\end{cases}
\]
More explicitly on the bipartitions appearing in \cref{thm:generalised k>j ss},
\[
\calf_{a,b}: ((ke+je),\varnothing) \mapsto ((ke+je+a,1^b),(a,1^b)) \text{ and}
\]
\[
\calf_{a,b}: ((ke+je-re,re-e+1),(e-1)) \longmapsto \begin{cases}
((ke+je-re+a,re-e+1+a),(e-1)) & \text{ if $b=0$,}\\
((ke+je-re+a,re-e+1+a,1^{b-1}),(e,1^b)) & \text{ if $b>0$.}
\end{cases}
\]
We define $-\calf_{a,b}$ as above but with all residue subscripts replaced by their negatives.
\end{defn}

\begin{eg}
Let $e=4$.
Then \cref{thm:k>jss} tells us that $\spe{((4),(4))} \cong \D{((4,1),(3))}\langle 1 \rangle \oplus \D{((8),\varnothing)}\langle 1 \rangle$.
We get from ${((4),(4))}$ to ${((6,1),(6,1))}$ by applying $\calf_{2,1}=\tilde{f}_3^{2} \tilde{f}_1^{2} \tilde{f}_0^{2}$, depicted as follows.
\[
\begin{tikzpicture}[scale=0.8,every node/.style={scale=0.8}]
\fill [] (2,0) node {$\xmapsto{\tilde{f}_0^{2}}$};
\fill [] (6,0) node {$\xmapsto{\tilde{f}_1^{2}}$};
\fill [] (10,0) node {$\xmapsto{\tilde{f}_3^{2}}$};
\fill [] (0,0) node {$\gyoung(0123,,0123)$};
\fill [] (4,0) node {$\gyoung(0123!\lc0,,!\wh0123!\lc0)$};
\fill [] (8,0) node {$\gyoung(01230!\lc1,,!\wh01230!\lc1)$};
\fill [] (12,0) node {$\gyoung(012301,!\lc3,,!\wh012301,!\lc3)$};
\end{tikzpicture}
\]
\cref{thm:generalised k>j ss} tells us that 
$\spe{((6,1),(6,1))} \cong \D{((6,3),(4,1))}\langle 1 \rangle \oplus \D{((10,1),(2,1))}\langle 1 \rangle$.
Now we focus on the first summand, letting $\mu = ((4,1),(3))$.
It follows from above that the first summand of $\spe{((6,1),(6,1))}$ is obtained by applying $\calf_{2,1}$ to $\mu$, depicted as follows.
\[
\begin{tikzpicture}[scale=0.8,every node/.style={scale=0.8}]
\fill [] (2,0) node {$\xmapsto{\tilde{f}_0^{2}}$};
\fill [] (6,0) node {$\xmapsto{\tilde{f}_1^{2}}$};
\fill [] (10,0) node {$\xmapsto{\tilde{f}_3^{2}}$};
\fill [] (0,0) node {$\gyoung(0123,3,,012)$};
\fill [] (4,0) node {$\gyoung(0123!\lc0,!\wh3!\lc0,,!\wh012)$};
\fill [] (8,0) node {$\gyoung(01230!\lc1,!\wh30!\lc1,,!\wh012)$};
\fill [] (12,0) node {$\gyoung(012301,301,,012!\lc3,3)$};
\end{tikzpicture}
\]
Similarly, we may apply $\calf_{2,3}= \tilde{f}_1^{4} \tilde{f}_2^{2} \tilde{f}_3^{2} \tilde{f}_0^{2}$ to $((4),(4))$ to obtain ${((6,1^3),(6,1^3))}$, depicted as follows.
\[
\begin{tikzpicture}[scale=0.8,every node/.style={scale=0.8}]
\fill [] (2,0) node {$\xmapsto{\tilde{f}_0^{2}}$};
\fill [] (6,0) node {$\xmapsto{\tilde{f}_3^{2}}$};
\fill [] (10,0) node {$\xmapsto{\tilde{f}_2^{2}}$};
\fill [] (14,0) node {$\xmapsto{\tilde{f}_1^{4}}$};
\fill [] (0,0) node {$\gyoung(0123,,0123)$};
\fill [] (4,0) node {$\gyoung(0123!\lc0,,!\wh0123!\lc0)$};
\fill [] (8,0) node {$\gyoung(01230,!\lc3,,!\wh01230,!\lc3)$};
\fill [] (12,0) node {$\gyoung(01230,3,!\lc2,,!\wh01230,3,!\lc2)$};
\fill [] (16,0) node {$\gyoung(01230!\lc1,!\wh3,2,!\lc1,,!\wh01230!\lc1,!\wh3,2,!\lc1)$};
\end{tikzpicture}
\]
Now \cref{thm:generalised k>j ss} tells us that 
$\spe{((6,1^3),(6,1^3))} \cong \D{((6,3,1^2),(4,1^3))}\langle 1 \rangle \oplus \D{((10,1^3),(2,1^3))}\langle 1 \rangle$,
so the first summand of $\spe{((6,1^3),(6,1^3))}$ is obtained by applying $\calf_{2,3}$ to $\mu$, depicted as follows.
\[
\begin{tikzpicture}[scale=0.8,every node/.style={scale=0.8}]
\fill [] (2,0) node {$\xmapsto{\tilde{f}_0^{2}}$};
\fill [] (6,0) node {$\xmapsto{\tilde{f}_3^{2}}$};
\fill [] (10,0) node {$\xmapsto{\tilde{f}_2^{2}}$};
\fill [] (14,0) node {$\xmapsto{\tilde{f}_1^{4}}$};
\fill [] (0,0) node {$\gyoung(0123,3,,012)$};
\fill [] (4,0) node {$\gyoung(0123!\lc0,!\wh3!\lc0,,!\wh012)$};
\fill [] (8,0) node {$\gyoung(01230,30,,012!\lc3,3)$};
\fill [] (12,0) node {$\gyoung(01230,30,!\lc2,,!\wh0123,3,!\lc2)$};
\fill [] (16,0) node {$\gyoung(01230!\lc1,!\wh30!\lc1,!\wh2,!\lc1,,!\wh0123,3,2,!\lc1)$};
\end{tikzpicture}
\]
\end{eg}

We next introduce and extend some notation from~\cite[Definition 3.1]{Sutton17II} in order to work with bihooks that are the transpose of those handled in \cref{thm:generalised k>j ss}.

\begin{defn}
For $x\in\bbn$, we define the following weakly decreasing sequence of $e-1$ non-negative integers summing to $x$.
\[
\{x\} := \left\lfloor \frac{x+e-2}{e-1} \right\rfloor, \left\lfloor \frac{x+e-3}{e-1} \right\rfloor, \dots, \left\lfloor \frac{x}{e-1} \right\rfloor
\]
If $\la = (\la_1, \la_2, \dots, \la_r)$ is a partition, then
we define $\{\la\} := (\{\la_1\},\{\la_2\},\dots,\{\la_r\})$, and analogously if $\la = (\la^{(1)}, \la^{(2)}) \in \mptn 2 n$, then $\{\la\}:= (\{\la^{(1)}\}, \{\la^{(2)}\})$.
\end{defn}

\begin{eg}
Let $e=3$, $\la = ((5e),\varnothing) = ((15),\varnothing)$ and $\mu = ((3e,e+1),(e-1)) = ((9,4),(2))$.
Then $\{\la\} = ((8,7),\varnothing)$ and $\{\mu\} = ((5,4,2^2),(1^2))$.
\end{eg}

\begin{cor}\label{cor:generalised k>j ss transpose}
Let $\la = ((b+1, 1^{je+a-1}), (b+1, 1^{ke+a-1}))$, for some $k\geq j \geq 1$, with $0< a \leq e$ and $0 \leq b < e$ with $a+b \neq e$, or for $a=b=0$ and set $n:=|\la|$.
Suppose that $p$ does not divide any of the integers $k+j, k+j-1, \dots, k-j+2$, or else that $p = 2 = j$ and $k\equiv 1 \pmod 4$.
Then,
\[
\spe\la \cong 
\bigoplus_{r=1}^{j} \D{-\calf_{a,b}\{\left( \left(n-re,(r-1)e+1\right),(e-1) \right)\}}\langle 2k+j \rangle \oplus \D{-\calf_{a,b}\{((n), \varnothing)\}} \langle 2k+j \rangle. 
\]
\end{cor}

\begin{proof}
We assume that $a=b=0$, and observe that by~\cite[Theorems 7.25 and 8.5]{kmr}, we have
\[
((\spe{((ke),(je))})^{\sgn})^\circledast \langle 2k{+}2j\rangle \cong \rspe{((1^{je}),(1^{ke}))}{^\circledast} \langle 2k{+}2j \rangle \cong \spe{((1^{je}),(1^{ke}))},
\]
and so
\[
\spe{((1^{je}),(1^{ke}))} \cong \bigoplus_{\nu} \D{m_{e,\kappa}(\nu)} \langle 2k{+}j \rangle,
\]
where the sum is over all $\nu$ appearing in \cref{thm:k>jss}.

Thus we must compute the images of the bipartitions $\nu$ under the Mullineux map.
First, it is simple to check that
\[
m_{e,\kappa} (((n),\varnothing)) = (\tilde{f}_1 \tilde{f}_2 \dots \tilde{f}_{-1} \tilde{f}_0)^{k+j} \varnothing
= \{((n),\varnothing)\}.
\]
Likewise, it's not so difficult to see that
\[
m_{e,\kappa} ( \left(((k+j-r)e,(r-1)e+1),(e-1)\right) ) = (\tilde{f}_1^2  \tilde{f}_2^2\dots \tilde{f}_{-1}^2 \tilde{f}_0^2)^{r} (\tilde{f}_1 \tilde{f}_2 \dots \tilde{f}_{-1} \tilde{f}_0)^{k+j-2r}\varnothing.
\]
Note that
\[
\tilde{f}_2^2 \dots \tilde{f}_{-1}^2 \tilde{f}_0^2 (\tilde{f}_1 \tilde{f}_2 \dots \tilde{f}_{-1} \tilde{f}_0)^{k+j-2r}\varnothing = \{(((k+j-2r+1)e-1),(1^{e-1}))\}
\]
and that applying $(\tilde{f}_1^2 \tilde{f}_2^2 \dots \tilde{f}_{-1}^2\tilde{f}_0^2 )^{r-1} \tilde{f}_1^2$ to this adds two conormal nodes, one in one of the first $e{-}1$ rows of the first component and the other in one of the next $e{-}1$ rows of the first component.

Now we note that $-\calf_{a,b}(((1^{je}),(1^{ke})))=\la$, and that $-\calf_{a,b}$ is well-behaved on all bipartitions we are applying it to in the statement of the corollary; by this, we mean that at each step we add the maximal number of conormal nodes, so that the corresponding sequence of $i$-induction functors sends a simple module $\D\nu$ to a simple module $\D{-\calf_{a,b}(\nu)}$.
The result follows.
\end{proof}

\begin{rem}
In the corollary above, and its proof, we essentially applied the twist by sign and then applied $-\calf_{a,b}$.
We could have done these the other way round, taking our bipartitions in \cref{thm:k>jss} and applying $\calf_{a,b}$ first, followed by the sign-twist.
\end{rem}

\subsection{Some non-semisimple cases: small $j$}\label{subsec:nonss}

Here, we further examine some of the non-semisimple cases, once again starting with the Specht modules $\spe{((ke),(je))}$.
Once the structures of these modules are determined, we can argue as in \cref{thm:generalised k>j ss,cor:generalised k>j ss transpose} to determine the structure of the other decomposable Specht modules that are obtained by applying divided power functors, and then a sign-twist.
We already completely determined the structure of $\spe{((ke),(je))}$ when $j=1$ in \cref{prop:j=1}, and for $j>1$ \cref{thm:k>jss} tells us exactly when $\spe{((ke),(je))}$ is not semisimple.
Of course, we have more bihooks corresponding to decomposable Specht modules for this `$j=1$ situation'.
We can apply the argument from \cref{thm:generalised k>j ss} to the $j=1$ case in \cref{prop:j=1} to yield the following.

\begin{cor}\label{cor:j=1induced}
Suppose $k\geq 1$, and let $\la = ((ke + a, 1^b), (e + a, 1^b))$, for some $k \geq 1$, with $0 < a \leq e$ and $0 \leq b < e$ with $a + b \neq e$, or for $a = b = 0$.
Then if $p\nmid k+1$, $\spe{\la}$ is semisimple, and moreover
\[
\spe\la \cong \begin{cases}
\D{((ke+e+a), (a))} \langle 1 \rangle \oplus \D{\left( \left(ke+a,a+1\right),(e-1) \right)}\langle 1 \rangle & \text{if } b = 0,\\
\\
\D{((ke+e+a,1^b),(a,1^b))} \langle 1 \rangle \oplus \D{\left( \left(ke+a,a+1, 1^{b-1}\right),(e,1^b) \right)}\langle 1 \rangle & \text{if } b \geq 1.
\end{cases}
\]
If $p \mid k+1$, then $\spe{\la}$ is non-semisimple, and moreover
\[
\spe\la \cong \begin{cases}
\D{((ke+e+a), (a))} \langle 1 \rangle \hspace{2pt} |\hspace{1pt} \D{\left( \left(ke+a,a+1\right),(e-1) \right)}\langle 1 \rangle  \hspace{2pt} |\hspace{1pt} \D{((ke+e+a), (a))} \langle 1 \rangle& \text{if } b = 0,\\
\\
\D{((ke+e+a,1^b),(a,1^b))} \langle 1 \rangle \hspace{2pt} |\hspace{1pt} \D{\left( \left(ke+a,a+1, 1^{b-1}\right),(e,1^b) \right)}\langle 1 \rangle \hspace{2pt} |\hspace{1pt} \D{((ke+e+a,1^b),(a,1^b))} \langle 1 \rangle & \text{if } b \geq 1.
\end{cases}
\]
\end{cor}

\begin{proof}
If $p\nmid k+1$, this is just a special case of \cref{thm:generalised k>j ss}.
If $p \mid k+1$, then we apply divided power functors as in the proof of \cref{thm:generalised k>j ss}, noting that they are exact functors that send simple modules to simple modules.

Thus in the non-semisimple cases we may apply the corresponding (composition of) functors to the module $\spe{((ke),(je))} \cong \D{((n),\varnothing)}\langle 1\rangle \hspace{2pt} |\hspace{1pt} \D{\left( \left(ke,1\right),(e-1) \right)}\langle 1\rangle \hspace{2pt} |\hspace{1pt} \D{((n),\varnothing)}\langle 1\rangle$ to obtain the result, using the following simple observation.

Note that if we have a module $M\cong A\hspace{2pt} |\hspace{1pt}B\hspace{2pt} |\hspace{1pt}A$, for simple modules $A$ and $B$, and an exact functor $\calf$ such that $\calf(A)$ and $\calf(B)$ are simple modules, then we may apply $\calf$ to the short exact sequences
\[
0 \rightarrow A \rightarrow M \rightarrow B\hspace{2pt} |\hspace{1pt}A \rightarrow 0 \quad \text{ and } \quad 0 \rightarrow A\hspace{2pt} |\hspace{1pt}B \rightarrow M \rightarrow A \rightarrow 0
\]
to yield 
\[
0 \rightarrow \calf(A) \rightarrow \calf(M) \rightarrow \calf(B\hspace{2pt} |\hspace{1pt}A) \rightarrow 0 \quad \text{ and } \quad 0 \rightarrow \calf(A\hspace{2pt} |\hspace{1pt}B) \rightarrow \calf(M) \rightarrow \calf(A) \rightarrow 0.
\]
If we additionally assume that $\calf$ preserves non-split extensions for all modules we are applying it to,
it is easy to see that $\calf(A\hspace{2pt} |\hspace{1pt}B) \cong \calf(A)\hspace{2pt} |\hspace{1pt}\calf(B)$ and $\calf(B\hspace{2pt} |\hspace{1pt}A) \cong \calf(B)\hspace{2pt} |\hspace{1pt}\calf(A)$.
It follows that $\calf(M)$ must have submodules isomorphic to $\calf(A)$ and $\calf(A)\hspace{2pt} |\hspace{1pt}\calf(B)$ and quotients isomorphic to $\calf(A)$ and $\calf(B)\hspace{2pt} |\hspace{1pt}\calf(A)$, and thus that $\calf(M) \cong \calf(A) \hspace{2pt} |\hspace{1pt} \calf(B) \hspace{2pt} |\hspace{1pt} \calf(A)$.
The result now follows by applying this (with $\calf$ being the appropriate composition of divided power functors, as in the proof of \cref{thm:generalised k>j ss}) to
\[
\spe{((ke),(je))} \cong \D{((n),\varnothing)}\langle 1\rangle \hspace{2pt} |\hspace{1pt} \D{\left( \left(ke,1\right),(e-1) \right)}\langle 1\rangle \hspace{2pt} |\hspace{1pt} \D{((n),\varnothing)}\langle 1\rangle,
\]
noting that
\[
\calf(\D{((n),\varnothing)}\langle 1\rangle) \cong
\begin{cases}
\D{((ke+e+a), (a))} \langle 1 \rangle & \text{ if } b=0,\\
\D{((ke+e+a,1^b),(a,1^b))} \langle 1 \rangle & \text{ if } b\geq1;
\end{cases}
\]
and 
\[
\calf(\D{\left( \left(ke,1\right),(e-1) \right)}\langle 1\rangle) \cong
\begin{cases}
\D{\left( \left(ke+a,a+1\right),(e-1) \right)}\langle 1 \rangle & \text{ if } b=0,\\
\D{\left( \left(ke+a,a+1, 1^{b-1}\right),(e,1^b) \right)}\langle 1 \rangle & \text{ if }b\geq 1;
\end{cases}
\]
and that the non-split extensions are by assumption preserved by our divided power functors.
If they were not preserved, we would have a split extension that we could then apply an appropriate composition of (exact) divided power functors $e_i^{(r)}$, that would yield $\spe{((ke),(je))}$ as a decomposable module, which is a contradiction.
The result now follows.
\end{proof}

As in \cref{cor:generalised k>j ss transpose} and the remark thereafter, we may twist our Specht modules by sign to obtain the result for the conjugate bipartitions.

\begin{cor}\label{cor:j=1induced transpose}
Suppose $k\geq 1$, and let $\la = ((b+1, 1^{e+a-1}), (b+1, 1^{ke+a-1}))$, for some $k \geq 1$, with $0 < a \leq e$ and $0 \leq b < e$ with $a + b \neq e$, or for $a = b = 0$.
Then $\spe\la$ is semisimple if and only if $p\nmid k+1$, and the structure of $\spe\la$ is given by applying the Mullineux map $m_{e,\kappa}$ to all of the bipartitions appearing in \cref{cor:generalised k>j ss transpose}.
\end{cor}

\begin{rem}
When $j=1$, observe that the set of all decomposable Specht modules indexed by bihooks coincides with those that are semisimple.
\end{rem}

The next natural case to handle is $j=2$, since this will leave us with a uniform condition for the non-semisimplicity in the remaining cases.
The following lemma will enable us to apply results for the Schur algebra to our Specht modules when they are not semisimple.

\begin{lem}\label{lem:charchange}
Over an arbitrary field \(\bbf\), we have \(\sL\mu \cong \scrm(\D{\scrt(\mu)}) \) for all two-column partitions \(\mu \in \scrp_n\).
\end{lem}
\begin{proof}
Consider the simple \(\scrr_{n\delta, \bbc}\)-module \(\D{\scrt(\mu),\bbc}\).
Following \cite[\S5.6]{bk09}, we select a \(\bbz\)-form \(\operatorname{M} \subset \D{\la, \bbc}\) such that \(\operatorname{M} \otimes_{\bbz} \bbc = \D{\scrt(\mu), \bbc}\), and define the \(\scrr_{n\delta, \bbf}\)-module \(\operatorname{J}_{\scrt(\mu)} = \operatorname{M} \otimes_{\bbz} \bbf\).
By \cite[Theorem 5.17]{bk09}, we have
\begin{align}\label{sd2}
[\operatorname{J}_{\scrt(\mu)}: \D{\nu, \bbf}] = 
\begin{cases}
1 & \textup{if }\nu = \scrt(\mu);\\
0 & \textup{if }\nu \ndomby \scrt(\mu).
\end{cases}
\end{align}
As \(\D{\scrt(\mu),\bbc} \cong \widehat{\scrm}(\sL\mu_\bbc)\), it follows from \cite[Lemma 7.1]{km17a} that all simple factors of \(J_{\scrt(\mu)}\) are of the form \(\widehat{\scrm}(\operatorname{L}(\gamma)_\bbf)\), where
\begin{align}\label{sd1}
[\operatorname{J}_{\scrt(\mu)} : \widehat{\scrm}(\sL{\gamma}_\bbf)]  = 
\begin{cases}
1 & \textup{if }\gamma = \mu;\\
0 & \textup{if }\gamma \ndomby \mu.
\end{cases}
\end{align}

Now, going by induction on dominance order on two-column partitions, we show that  \(\widehat{\scrm}(\sL{\mu}_\bbf) \cong \D{\scrt(\mu),\bbf}\).
Make the induction assumption on \(\mu\).
By (\ref{sd1}), the simple factors of \(\operatorname{J}_{\scrt(\mu)}\) are \(\widehat{\scrm}(\sL{\mu}_\bbf)\) and modules of the form \(\widehat{\mathscr{M}}(\sL{\gamma}_\bbf) \cong \D{\mathscr{T}(\gamma),\bbf}\) for \(\gamma \domsby \mu\).
But we have \(\D{\scrt(\mu),\bbf} \not \cong \D{\scrt(\gamma),\bbf}\) when \(\gamma \domsby \mu\), so it follows from (\ref{sd2}) that  \(\widehat{\scrm}(\sL{\mu}_\bbf) \cong \D{\scrt(\mu),\bbf}\), as required.
\end{proof}

\begin{rem}
If we combine \cref{thm:decompnos,lem:moritaequiv,lem:charchange,lem:Pierifilt},
we now have a formula for readily computing all composition factors of $\spe{((ke),(je))}$, and their multiplicities. We remark that by \cref{indprodS} and \cref{KMWords1}, all simple factors of $\spe{((ke),(je))}$ are relatively unshifted with respect to each other; i.e., if one simple is shifted by $j$, all shifts must be by $j$.
\end{rem}

\begin{prop}\label{prop:j=2}
Suppose $k \geq 2$, and let $n=ke+2e$.
Then
\begin{enumerate}[label=(\roman*)]
\item if $p\neq 2$ and $p$ does not divide any of the integers $k+2$, $k+1$, $k$, or if $p=2$ and $k \equiv 1 \pmod 4$, then
\[
\spe{((ke),(2e))} \cong \D{\left( \left(ke,e+1\right),(e-1) \right)}\langle 2\rangle \oplus\D{\left( \left(ke+e,1\right),(e-1) \right)}\langle 2\rangle \oplus \D{((n),\varnothing)}\langle 2\rangle;
\]

\item if $2\neq p \mid k+2$, then
\[
\spe{((ke),(2e))} \cong \D{\left( \left(ke,e+1\right),(e-1) \right)}\langle 2\rangle
\oplus
\left(\D{((n),\varnothing)}\langle 2\rangle \hspace{2pt} |\hspace{1pt} \D{\left( \left(ke+e,1\right),(e-1) \right)}\langle 2\rangle \hspace{2pt} |\hspace{1pt} \D{((n),\varnothing)}\langle 2\rangle \right);
\]

\item if $2\neq p \mid k+1$, or if $p=2$ and $k \equiv 3\pmod 4$ then
\[
\spe{((ke),(2e))} \cong \D{\left( \left(ke+e,1\right),(e-1) \right)}\langle 2\rangle
\oplus
\left(\D{((n),\varnothing)}\langle 2\rangle \hspace{2pt} |\hspace{1pt} \D{\left( \left(ke,e+1\right),(e-1) \right)}\langle 2\rangle \hspace{2pt} |\hspace{1pt} \D{((n),\varnothing)}\langle 2\rangle \right);
\]

\item if $2\neq p \mid k$, then
\[
\spe{((ke),(2e))} \cong \D{((n),\varnothing)}\langle 2\rangle
\oplus
\left( \D{\left( \left(ke+e,1\right),(e-1) \right)}\langle 2\rangle \hspace{2pt} |\hspace{1pt} \D{\left( \left(ke,e+1\right),(e-1) \right)}\langle 2\rangle \hspace{2pt} |\hspace{1pt}  \D{\left( \left(ke+e,1\right),(e-1) \right)}\langle 2\rangle \right);
\]

\item if $p=2$ and $k \equiv 0 \pmod 4$, then
\begin{multline*}
\spe{((ke),(2e))} \cong \D{((n),\varnothing)}\langle 2\rangle
\oplus\\
\left(\D{\left( \left(ke+e,1\right),(e-1) \right)}\langle 2\rangle \hspace{2pt} |\hspace{1pt}
\D{((n),\varnothing)}\langle 2\rangle \hspace{2pt} |\hspace{1pt} \D{\left( \left(ke,e+1\right),(e-1) \right)}\langle 2\rangle
\hspace{2pt} |\hspace{1pt} \D{((n),\varnothing)}\langle 2\rangle \hspace{2pt} |\hspace{1pt} \D{\left( \left(ke+e,1\right),(e-1) \right)}\langle 2\rangle\right);
\end{multline*}

\item if $p=2$ and $k \equiv 2 \pmod 4$, then $\spe{((ke),(2e))}$ is indecomposable, and
\[
\begin{tikzpicture}
\node at (-3,1) {$\spe{((ke),(2e))} \langle -2 \rangle \cong$};

\node (triv1) at (0,0) {$\D{((n),\varnothing)}$};

\node (l11) at (0,2) {$\D{((ke+e,1),(e-1))}$};

\node (l2) at (3.5,1) {$\D{((ke,e+1),(e-1))}$};

\node (l12) at (7,0) {$\D{((ke+e,1),(e-1))}$};

\node (triv2) at (7,2) {$\D{((n),\varnothing)}$};

\draw (triv1.north) -- (l11.south);

\draw (l11.south east) -- (l2.north west);

\draw (l2.south east) -- (l12.north west);

\draw (l12.north) -- (triv2.south);
\end{tikzpicture}
\]
\end{enumerate}
\end{prop}

\begin{proof}
Part (i) is completely handled in \cref{thm:k>jss}.
For part (ii), we apply \cref{thm:decompnos} to $\scrm(\spe{((ke),(2e))}) \cong \Delta(1^k) \otimes \Delta(1^2)$, which has a filtration by the Weyl modules $\Delta(1^{k+2}) = \sL{1^{k+2}}$, $\Delta(2,1^k)$, and $\Delta(2^2,1^{k-2})$.
We may check that $\Delta(2,1^k) \cong \sL{1^{k+2}} \hspace{2pt} |\hspace{1pt} \sL{2,1^k}$ and $\Delta(2^2,1^{k-2}) = \sL{2^2,1^{k-2}}$.
Given that $\spe{((ke),(2e))}$ is decomposable and that each summand is self-dual, it follows that
\[
\Delta(1^k) \otimes \Delta(1^2) \cong \sL{2^2,1^{k-2}} \oplus \left( \sL{1^{k+2}} \hspace{2pt} |\hspace{1pt} \sL{2,1^k} \hspace{2pt} |\hspace{1pt} \sL{1^{k+2}} \right).
\]
Since $\scrm(\D{((n),\varnothing)}) = \sL{1^{k+2}}$,
$\scrm(\D{( (ke+e,1),(e-1) )}) =  \sL{2,1^k}$,
and $\scrm(\D{( (ke,e+1),(e-1) )}) = \sL{2^2,1^{k-2}}$, by \cref{lem:charchange}, the ungraded result follows.

To see the grading shifts, we note that $\spe{((ke+e),(e))} \cong \D{((n),\varnothing)}\langle 1\rangle \hspace{2pt} |\hspace{1pt} \D{\left( \left(ke+e,1\right),(e-1) \right)}\langle 1\rangle \hspace{2pt} |\hspace{1pt} \D{((n),\varnothing)}\langle 1\rangle$ by \cref{thm:k>jss}.
It is easy to check that $\dim (\im\alpha_{k,2})>1$.
Indeed, $\psi_e$ kills the `leading' term $(k+1) z_{((ke),(2e))}$ of $\alpha_{k,2}(z_{((ke+e),(e))})$, while mapping all other terms in $\alpha_{k,2}(z_{((ke+e),(e))})$ to new basis vectors that are more dominant, with coefficients $k, k-1, \dots, 1$, at least some of which are nonzero modulo $p$.
It easily follows that $\alpha_{k,2}$ is injective: $\im \alpha_{k,2}$ must have simple head $\D{((n),\varnothing)}\langle 2\rangle$, and, up to grading shifts, it must be a submodule of $\D{((n),\varnothing)} \hspace{2pt} |\hspace{1pt} \D{\left( \left(ke+e,1\right),(e-1) \right)} \hspace{2pt} |\hspace{1pt} \D{((n),\varnothing)}$.
\cref{prop:alphahom,prop:gammahom} thus give us the required shifts for all simple factors.

Parts (iii) and (iv) are almost identical, and the details of the ungraded result are left to the reader.
For the grading, we now have that $\spe{((ke+e),(e))} \cong \D{((n),\varnothing)}\langle 1\rangle \oplus  \D{\left( \left(ke+e,1\right),(e-1) \right)}\langle 1\rangle$ in both cases, so that \cref{prop:alphahom} provides us the necessary grading shift for the non-trivial simple submodule (which is a summand in the case of part (iii), but not for part (iv)).
It is difficult to directly check that the homomorphism $\alpha_{k,2}$ is injective in this case, so a priori, it may kill the trivial factor.
Applying \cref{prop:gammahom} gives the grading shift for $\D{((ke,e+1),(e-1))}$.
In part (iii), this also gives us the grading shift on the simple submodule, and the grading shift for the remaining simple factor (the one-dimensional simple head of the non-simple summand) follows by \cref{prop:selfdualspecht}, as in the proof of \cref{prop:j=1}.
In part (iv), \cref{prop:gammahom} gives the grading shift for $\D{((ke,e+1),(e-1))}$ and the non-trivial simple module below it, from which we can once again deduce the grading shift on the simple head of the non-simple summand.
In order to obtain the grading shift on the trivial summand in this case, we note that $\alpha_{k,2}\circ \alpha_{k+1,1}: \spe{((n),\varnothing)} \rightarrow \spe{((ke),(2e))}$ is a non-zero degree 2 homomorphism -- this can be seen by noting that the leading term is $(k+2)(k+1) z_{((ke),(2e))}$, and that this cannot appear when reducing any of the products of $\Psi$ terms that arise when composing the two homomorphisms.

For part (v), one can check using \cref{thm:Henkeformula,lem:no.indecompsummands} that the Schur functor maps $\Delta(1^k) \otimes \Delta(1^2)$ to $\sM{k,2}$, which has two indecomposable summands when $p=2$ and $k\equiv 0 \pmod 4$.
Furthermore, by \cref{thm:decompnos}, in this case $\Delta(2,1^{k})$ has a copy of $\sL{1^{k+2}}$ as a submodule, and $\Delta(2^2,1^{k-2})$ has both $\sL{2,1^{k}}$ and $\sL{1^{k+2}}$ appearing as composition factors.
There is a unique way to combine these modules to make one with two summands, each of which is self-dual.
Namely,
\[
\Delta(1^k) \otimes \Delta(1^2) \cong \sL{1^{k+2}}
\oplus
\left(\sL{2,1^{k}} \hspace{2pt} |\hspace{1pt} \sL{1^{k+2}} \hspace{2pt} |\hspace{1pt} \sL{2^2,1^{k-2}} \hspace{2pt} |\hspace{1pt} \sL{1^{k+2}} \hspace{2pt} |\hspace{1pt} \sL{2,1^{k}}\right).
\]
Matching labels as in \cref{thm:k>jss} and applying \cref{lem:charchange} yields the ungraded version of the stated result.
For the grading shifts, \cref{prop:alphahom,prop:gammahom} give us the required information as in part (ii).

In part (vi), one can check using \cref{thm:Henkeformula} that $\sM{k,2}$ is indecomposable when $p=2$ and $k\equiv 2 \pmod 4$.
Furthermore, by \cref{thm:decompnos}, in this case $\Delta(2,1^{k})$ has a copy of $\sL{1^{k+2}}$ as a submodule, and $\Delta(2^2,1^{k-2})$ has a copy of $\sL{2,1^{k}}$ as a submodule.
With the fact that $\Delta(1^k) \otimes \Delta(1^2)$ has a filtration by the Weyl modules $\Delta(1^{k+2})$, $\Delta(2,1^{k})$, and $\Delta(2^2,1^{k-2})$, it follows that 
there are only two ways to construct this indecomposable self-dual module with these factors, namely
\[
\begin{tikzcd}
	& {\sL{2,1^k}} &&& {\sL{2,1^k}} && {\sL{1^{k+2}}} \\
	{\sL{1^{k+2}}} & {\sL{2^2,1^{k-2}}} & {\sL{1^{k+2}}} & {\text{or}} && {\sL{2^2,1^{k-2}}} \\
	& {\sL{2,1^k}} &&& {\sL{1^{k+2}}} && {\sL{2,1^k}}
	\arrow[no head, from=3-5, to=1-5]
	\arrow[no head, from=1-7, to=3-7]
	\arrow[no head, from=1-5, to=2-6]
	\arrow[no head, from=2-6, to=3-7]
	\arrow[no head, from=1-2, to=2-2]
	\arrow[no head, from=1-2, to=2-1]
	\arrow[no head, from=1-2, to=2-3]
	\arrow[no head, from=2-1, to=3-2]
	\arrow[no head, from=2-2, to=3-2]
	\arrow[no head, from=2-3, to=3-2]
\end{tikzcd}
\]
By \cite[Lemma~A3.1]{Donkin}, $\Delta(1^k) \otimes \Delta(1^2)$ must have $\Delta(1^{k+2}) = \sL{1^{k+2}}$ appearing as a quotient module, which rules out the first configuration above.
The ungraded result follows as before.
The gradings follow from \cref{prop:alphahom,prop:gammahom} as before.
\end{proof}

\begin{rem}
For the Specht modules in parts (v) and (vi) of the above proposition, the second and third authors did not determine in \cite{ss20} whether or not they were decomposable, so that even this coarser information is new.
\end{rem}

Applying the map $\calf_{a,b}$ to the labels of the above decomposable Specht modules, the observation in the proof of \cref{cor:j=1induced} immediately yields the following.

\begin{cor}\label{cor:j=2induced}
Suppose $k \geq 2$, and let $n=ke+2e$.
For $\la = ((ke + a, 1^b), (2e + a, 1^b))$, for some $k \geq 2$, with $0 < a \leq e$ and $0 \leq b < e$ with $a + b \neq e$, or for $a = b = 0$, the structure of $\spe\la$ is obtained from \cref{prop:j=2} by replacing each simple module $\D\mu\langle 2 \rangle$ therein with $\D{\calf_{a,b}(\mu)}\langle 2 \rangle$, and the structure of $\spe{\la'}$ is obtained from \cref{prop:j=2} by replacing each simple module $\D\mu \langle 2 \rangle$ therein with $\D{-\calf_{a,b}(\{\mu\})}\langle 2 \rangle = \D{m_{e,\kappa}\circ\calf_{a,b}(\mu)}\langle 2 \rangle$.
\end{cor}

\subsection{Decomposability in characteristic 2}\label{subsec:p=2}

With the Specht modules for $j=1$ and $j=2$ completely understood, we will shift away from examining cases of small $j$.
First, we will build on parts (v) and (vi) of \cref{prop:j=2}.
Observe that the condition in the next result is remarkably similar to that of Murphy~\cite[Theorem 4.5]{gm80}, which determines the decomposability of level 1 Specht modules indexed by hooks.

\begin{prop}\label{prop:p=2decomposables}
Let $k\geq j \geq 1$, $p=2$, and let $l$ be such that $2^{l-1} \leq j < 2^l$.
Then $\spe{((ke),(je))}$ is decomposable if and only if $k \not\equiv j \pmod{2^l}$.
\end{prop}

\begin{proof}
Using the Morita equivalence $\scrm$ as before, it suffices to show that $\sM{k,j} \cong \sY{k,j}$ if and only if $k\equiv j \pmod{2^l}$.
In other words, we will show that if $0\leq m < j$, some $\sY{k+j-m,m}$ is a summand of $\sM{k,j}$ if and only if $k \not\equiv j \pmod{2^l}$.
We use \cref{thm:Henkeformula,lem:no.indecompsummands} at length, and we proceed to determine if $j-m \leq_2 k+j-2m$.

Set $m = j - 2^a$ for $a \in \{0,1, \dots,  l-1\}$.
In order for $\sY{k+j-m,m}$ to \emph{not} be a summand of $\sM{k,j}$, we must have $k-j \equiv 0,1, \dots, 2^a-1 \pmod{2^{a+1}}$.
In order for this to hold simultaneously for all $a \in \{0,1, \dots,  l-1\}$, we must have $k-j \equiv 0 \pmod{2^l}$.
We thus have shown that if $k \not\equiv j \pmod{2^l}$, then $\spe{((ke),(je))}$ is decomposable.

It remains to show that when $k \equiv j \pmod{2^l}$, no $\sY{k+j-m,m}$ appears as a summand of $\sM{k,j}$, for $m$ not of the form considered above.
If $k \equiv j \pmod{2^l}$, then $k+j-2m \equiv 2(j-m) \pmod{2^l}$.
It follows immediately that $j-m \not\leq_2 k+j-2m$ 
and the result follows.
\end{proof}

By applying $i$-induction functors, we deduce the following corollary.

\begin{cor}\label{cor:p=2decomposables}
Let $p=2$, $\la= ((ke + a, 1^b), (je + a, 1^b))$ or $((b + 1, 1^{je+a-1}), (b + 1, 1^{ke+a-1}))$, for some $j, k > 1$, $0 < a \leq e$ and $0 \leq b < e$ with $a + b \neq e$, or for $a = b = 0$, and let $l$ be such that $2^{l-1} \leq j < 2^l$.
Then $\spe\la$ is decomposable if and only if $k \not\equiv j \pmod{2^l}$.
\end{cor}

\begin{rem}
The above corollary fills a gap in \cite[Theorem 4.1]{ss20}, where we were unable to determine the decomposability of $\spe\la$ when $p=2$ and $j+k$ is even.
This also strengthens \cite[Conjecture 4.2]{ss20}: if $e\neq2$, we conjecture that \cite[Theorems 3.8 and 4.1]{ss20} and \cref{cor:p=2decomposables} provide a complete list of decomposable Specht modules indexed by bihooks.
\end{rem}

\subsection{Almost-semisimple Specht modules}\label{subsec:almostss}

Next, we examine cases that are `close to being semisimple', in the sense that every direct summand contains few simple factors.
We saw in \cref{thm:k>jss} that $\spe{((ke),(je))}$ is semisimple when $p$ does not divide any of the integers $k+j,k+j-1,\dots,k-j+2$.
The next case we look at is when $p$ divides exactly one of these integers.

\begin{prop}\label{twopartdecompnumbers}
Let $k \geq j > 1$, $m\in\{0,\dots,j\}$, $\la = (2^m,1^{n-2m})$ where $n=k+j$, and suppose that $p$ divides exactly one of the integers $n, n-1, \dots, n-2j+2$, so that $n \in \{ap, ap+1, \dots, ap+2j-2\}$ for some $a\in\bbn$.
We write $n = ap +i$ for $i\in \{0,1,\dots, 2j-2\}$.

\begin{enumerate}[label=(\roman*)]
\item If $(m,j) \neq (p,p)$, then
\[
\Delta(\la) \cong
\begin{cases}
\sL\la & \text{if } m\in\{0,1,\dots,\lfloor \frac{i+1}{2} \rfloor\} \cup \{i+2,i+3,\dots,j \}; \\
\sL{2^{i-m+1},1^{n -2i +2m -2}} \hspace{2pt} |\hspace{1pt} \sL\la
& \text{if } m\in\{\lfloor \frac{i+1}{2} \rfloor + 1, \dots, i+1\}.
\end{cases}
\]

\item If $i = j-1$ and $(m,j) = (p,p)$, then
\begin{align*}
\Delta(\la) \cong
\begin{cases}
\sL\la & \text{if } k+1\equiv p \pmod{p^2}; \\
\sL{1^n} \hspace{2pt} |\hspace{1pt} \sL\la& \text{if } k+1\equiv 0 \pmod{p^2}.
\end{cases}
\end{align*}
\end{enumerate}
\end{prop}

\begin{proof}
Let $\mu=(2^l,1^{n-2l})$ for $0\leq l \leq m$ and compute $[\Delta(\la):\sL\mu]$.
Since $[\Delta(\la):\sL\la]=1$, we suppose that $l<m$ and apply \cref{thm:decompnos} to determine $[\Delta(\la):\sL\mu]$, which is either $0$ or $1$.
	
Since $p$ divides exactly one of $n, n-1, \dots, n-2j+2$, it follows that $p \geq j \geq m >l \geq 0$.
Thus, with the exception of the special case where $p=j=m$, $l=0$, we have that $\lfloor \frac{m-l}{p} \rfloor =0 \preccurlyeq_p \lfloor \frac{n-2l+1}{p} \rfloor$ and $p \nmid m-l$.
Hence $[\Delta(\la):\sL\mu] = 1$ if and only if $p \mid n-m-l+1$.
The exceptional case only occurs when $p \mid k+1$ (so that $i = j-1$), and there we are looking at whether $\sL{1^n}$ occurs as a composition factor of $\Delta(2^j,1^{n-2j})$.
Applying \cref{thm:decompnos} again, we see that in this case it is a composition factor precisely when $1 \preccurlyeq_p a+1$, or in other words exactly when $p\mid a$.
Since $k+1 = ap$ here, this is equivalent to $p^2 \mid (k+1)$.
This situation occurs in case (ii) of the proposition statement, which completes the proof of that case.
We now treat the remaining case.

First suppose that $i \in \{0,1, \dots, j-2$\}.
Then
\[
i-m-l+1 \leq i-m+1 \leq j - 1 - m < j - m <  j \leq p
\]
and
\[
i-m-l+1 \geq i-2m+2 \geq -2m+2 \geq -2j+2 > -2p.
\]
It follows that $-2p < i-m-l+1 < p$, and thus $p \mid n-m-l+1$ if and only if $l = i-m+1$ or $l = i-m+1+p$.
Since $i\in \{0,1, \dots, j-2\}$, the condition that $p$ divides $k+j-i$ and no other integers in $[k-j+2, k+j]$ implies that $p\geq 2j - i -1$, and hence that $i-m+1+p \geq m$.
Since $0\leq l < m$, it follows that $p \mid n-m-l+1$ can only happen when $l = i-m+1$ and $m \in \{\lfloor\frac{i+1}{2}\rfloor + 1, \dots, i+1\}$.

Now suppose that $i \in \{j-1, j, \dots, 2j-2 \}$.
Then
\[
i-m-l+1 < i+1 \leq 2j-1 < 2p
\]
and
\[
i-m-l+1 \geq j-m-l \geq -l > -p.
\]
So $-p < i -m -l +1 < 2p$, and thus $p \mid n-m-l+1$ if and only if $l = i-m+1$ or $l = i-m+1-p$.
Since $i\in \{j-1,j, \dots, 2j-2\}$, the condition that $p$ divides $k+j-i$ and no other integers in $[k-j+2, k+j]$ implies that $p\geq i +1$, and hence that $i-m+1-p \leq -m \leq -1$.
Once again, we use that $0\leq l < m$ to see that $p \mid n-m-l+1$ can only happen when $l = i-m+1$ and $m \in \{\lfloor\frac{i+1}{2}\rfloor + 1, \dots, i+1\} \cap \{0,1,\dots,j\} = \{\lfloor\frac{i+1}{2}\rfloor + 1, \dots, j\}$.

Hence for any $i\in \{0,1,\dots, 2j-2\}$, we have $[\Delta(\la):\sL\mu] = 1$ if and only if $l=i-m+1$ and $m\in \{\lfloor\frac{i+1}{2}\rfloor + 1, \dots, i+1\}$.
\end{proof}

\begin{rem}
Note that the case $(m,j)=(p,p)$ \emph{only} occurs when $i=j-1$, which follows from the lower bounds on $p$ given in the above proof.
\end{rem}

Next, we decompose $\sM{n-j,j}$ into its indecomposable summands.

\begin{prop}\label{prop:noofsummands}
Let $k\geq j >1$, $k+j= n =ap+i$ for $i\in\{0,\dots,2j-2\}$, and assume that $p$ divides exactly one of $n,n-1,\dots,n-2j+2$.
Then
\[
\sM{n-j,j} \cong \bigoplus_{m=0}^{i-j} \sY{n-m,m}
\oplus \bigoplus_{m=\lceil \frac{i+1}{2}\rceil}^j \sY{n-m,m},
\]
unless $j=p$, $i=j-1$ (so that $p\mid k+1$), and $k+1 \equiv p \pmod{p^2}$, in which case
\[
\sM{n-j,j} \cong \sY{n}
\oplus \bigoplus_{m=\lceil \frac{i+1}{2}\rceil}^j \sY{n-m,m}.
\]

\end{prop}

\begin{proof}	
We apply \cref{thm:Henkeformula} to determine if $\sY{n-m,m}$ is a direct summand of $\sM{n-j,j}$ for all $m\in\{0,\dots,j\}$.

It is clear that $\sY{n-j,j}$ is a summand of $\sM{n-j,j}$, so we suppose that $m<j$.
Now, $1 \leq j-m \leq j  \leq p$.
Then the $p$-adic expansion of $j-m$ is $[j-m,0,0,\dots]$ except for when $j=p$ (and thus $p \mid k+1$ and $i = j-1$) and $m=0$, in which case $j-m$ has $p$-adic expansion $[0,1,0,0,\dots]$.

If $i\in \{0,1, \dots, j-1\}$, then the condition that $p$ divides $k+j-i$ and no other integers in $[k-j+2, k+j]$ implies that $p\geq 2j - i -1$, and therefore that
\[
p+i-2m \geq 2j-1 -2m \geq 1.
\]
Similarly, if $i\in \{j-1, j, \dots, 2j-2\}$, we have that $p\geq i+1$, and therefore
\[
p+i-2m \geq 2i+1 -2m \geq 2(i-j)+3.
\]
Thus if $i\in \{0,1,\dots, 2j-2\}$ and $m\in \{i+1,i+2,\dots, j-1\}$, we have $0 < p+i-2m \leq p-2$.
So the $p$-adic expansion of $p+i-2m$ is $[p+i-2m,0,0,\dots]$, and $j-m \leq_p p+i-2m \leq_p n-2m$.
It follows that $\sY{n-m,m}$ is always a summand of $\sM{n-j,j}$ when $m\in \{i+1,i+2,\dots, j-1\}$.

Next, suppose that $i\in \{0,1,\dots, j-1\}$ and $m \in \{0,1,\dots, i\}$.
Then $0 \leq i-m \leq i <j$, so $-m \leq i-2m < j-m \leq p$.

If $m\leq \lfloor i/2 \rfloor$, then $0\leq i-2m$ so that the $p$-adic expansion of $i-2m = n-ap -2m$ is $[i-2m,0,0,\dots]$, and it follows that $j-m \not\leq_p n-2m$, except in the case that $j=p$ and $m=0$.
In this exceptional case, we must look at the second entry in the $p$-adic expansion of $n$.
In particular, if $k+1 \equiv 0 \pmod{p^2}$, then $n \equiv p-1 \pmod{p^2}$, so that the first two entries of the $p$-adic expansion of $n$ are $p-1$ and $0$, and thus $j-m \not\leq_p n-2m$ again.
But if $k+1 \equiv p \pmod{p^2}$, then $n \equiv 2p-1 \pmod{p^2}$, so that the first two entries of the $p$-adic expansion of $n$ are $p-1$ and $1$, and thus $j-m \leq_p n-2m$ in this case.
This is the exceptional case in the statement of the proposition.

If $m > i/2$, then $1 \leq p+ i-2m < p$, so that the $p$-adic expansion of $p + i-2m = n-(a-1)p -2m$ is $[p+i-2m,0,0,\dots]$.
Since $j \leq p \leq p+i-m$, it follows that $j-m \leq p+i - 2m$, so that $j-m \leq_p p + i-2m \leq_p n-2m$.

Finally, suppose that $i\in \{j, j+1, \dots, 2j-2\}$ and $m \in \{0,1,\dots, i\}$.
Then, as above, $p\geq i+1$, so that we have
\[
-p < -i \leq i-2m \leq i < p \quad \text{ or, equivalently, } \quad 0 < p+i-2m < 2p.
\]
If $m \leq i/2$, then the above becomes $0 \leq i-2m < p$, so that the $p$-adic expansion of $i-2m$ is $[i-2m,0,0,\dots]$, so that $j-m \leq_p n - 2m$ if and only $j-m \leq_p i-2m$, if and only if $m \leq i-j$.
If instead, $m > i/2$, then the above becomes $0 < p+ i - 2m < p$, and the result follows as before.
This completes the proof.
\end{proof}

We now combine the above results to obtain the decomposition of the tensor product of Weyl modules. In the following theorem, we abuse notation and write $\sLx{x}$ to mean $\sL{2^x,1^{n-2x}}$.

\begin{thm}\label{thm:Schuralmostss}
Let $k\geq j >1$ and suppose that $p$ divides exactly one of the integers $k+j, k+j-1, \dots, k-j+2$, so that $n = k+j \in \{ap, ap+1, \dots, ap+2j-2\}$ for some $a\in\bbn$.
We write $n = ap +i$ for $i\in \{0,1,\dots, 2j-2\}$.
For $r \in \{0,1,  \dots, \lfloor i/2 \rfloor\}$, define the uniserial module 
\[
N_r := \sLx{\lfloor i/2 \rfloor - r} \hspace{2pt} |\hspace{1pt} \sLx{\lceil i/2 \rceil + 1 + r} \hspace{2pt} |\hspace{1pt} \sLx{\lfloor i/2 \rfloor - r}.
\]
First we assume that we are not in the case that $j=p$ and $k+1 \equiv p \pmod{p^2}$.
\begin{enumerate}[label=(\roman*)]
	\item Suppose that $i \in \{0,1,\dots, j-1\}$.
		\begin{enumerate}
			\item If $i$ is odd, then
			\[
			\Delta(k) \otimes \Delta(j) \cong
			\sLx{\frac{i+1}{2}}
			\oplus
			\bigoplus_{r=0}^{\lfloor \frac{i}{2} \rfloor} N_r
			\oplus
			\bigoplus_{m=i+2}^{j} \sLx{m}
			\]
			\item If $i$ is even, then
			\[
			\Delta(k) \otimes \Delta(j) \cong
			\bigoplus_{r=0}^{\lfloor \frac{i}{2} \rfloor} N_r
			\oplus
			\bigoplus_{m=i+2}^{j} \sLx{m}
			\]
		\end{enumerate}
	\item Suppose that $i\in \{j,j+1, \dots, 2j-2\}$.
		\begin{enumerate}
			\item If $i$ is odd, then
			\[
			\Delta(k) \otimes \Delta(j) \cong
			\bigoplus_{m=0}^{i-j} \sLx{m}
			\oplus
			\sLx{\frac{i+1}{2}}
			\oplus
			\bigoplus_{r=0}^{j - \lceil \frac{i}{2} \rceil -1} N_r
			\]
			\item If $i$ is even, then
			\[
			\Delta(k) \otimes \Delta(j) \cong
			\bigoplus_{m=0}^{i-j} \sLx{m}
			\oplus
			\bigoplus_{r=0}^{j - \lceil \frac{i}{2} \rceil -1} N_r
			\]
		\end{enumerate}
\end{enumerate}

Finally, in the exceptional case that $j=p$, $i=j-1$, $k+1 \equiv p \pmod{p^2}$. Then
	\[
	\displaystyle{
		\Delta(k) \otimes \Delta(j) \cong }
	\begin{cases}
	\displaystyle{
		\sLx0 \oplus \sLx1 \oplus \sLx2 }
	&\text{ if $j=p=2$,}\\
	\displaystyle{
		\sLx0 \oplus
		\bigoplus_{r=0}^{\lfloor \frac{i}{2} \rfloor - 1} N_r
		\oplus \sLx{j} }
	&\text{ if $j=p>2$.}
	\end{cases}
	\]
\end{thm}

\begin{proof}
We know from \cref{twopartdecompnumbers} that the Weyl module $\Delta(\la)$ is simple if $m\in\{0,1,\dots,\lfloor \frac{i+1}{2}\rfloor\} \cup\{i+2,i+3,\dots,j\}$, whereas if $m\in\{\lfloor \frac{i+1}{2}\rfloor+1,\dots,i+1\}$, then $\Delta(\la) \cong \sL{\mu} \hspace{2pt} |\hspace{1pt} \sL{\la}$, where $\mu = (2^{i-m+1},1^{n-2i+2m-2})$.
Moreover, we obtain a self-dual module by stacking the Weyl modules $\Delta(\la)$ and $\Delta(\mu) = \sL{\mu}$ to give $\sL{\mu} \hspace{2pt} |\hspace{1pt} \sL{\la} \hspace{2pt} |\hspace{1pt} \sL{\mu}$.
Examining the five cases in turn, along with knowing the number of summands we must obtain -- by \cref{lem:no.indecompsummands,prop:noofsummands} -- the result follows.
\end{proof}

Applying the map $\scrt$ from \cref{def:moritamaponlabels} and the observation in the proof of \cref{cor:j=1induced} immediately yields the following.

\begin{thm}\label{thm:pdiv1}
Let $k\geq j >1$ and suppose that $p$ divides exactly one of the integers $k+j, k+j-1, \dots, k-j+2$, so that $n = k+j \in \{ap, ap+1, \dots, ap+2j-2\}$ for some $a\in\bbn$.
We write $n = ap +i$ for $i\in \{0,1,\dots, 2j-2\}$.
\begin{enumerate}[label=(\roman*)]
	\item Then the structure of $\spe{((ke),(je))}$ is obtained from \cref{thm:Schuralmostss} by replacing each simple module $\sL\mu$ therein with $\D{\scrt (\mu)}\langle j \rangle$.
	\item Let $\la = ((ke + a, 1^b), (je + a, 1^b))$, for some $k \geq j\geq 1$, with $0 < a \leq e$ and $0 \leq b < e$ with $a + b \neq e$, or for $a = b = 0$. Then the structure of $\spe\la$ is obtained from \cref{thm:Schuralmostss} by replacing each simple module $\sL\mu$ therein with $\D{\calf_{a,b}\circ\scrt (\mu)}\langle j \rangle$
	\item Let $\la = ((b+1,1^{je+a-1}), (b+1,1^{ke+a-1}))$, for some $k \geq j\geq 1$, with $0 < a \leq e$ and $0 \leq b < e$ with $a + b \neq e$, or for $a = b = 0$. Then the structure of $\spe{\la}$ is obtained from \cref{thm:Schuralmostss} by replacing each simple module $\sL\mu$ therein with $\D{-\calf_{a,b}(\{\scrt (\mu)\})}\langle j \rangle = \D{m_{e,\kappa}\circ\calf_{a,b}\circ\scrt (\mu)}\langle j \rangle$.
\end{enumerate}
\end{thm}

We finish with an example to illustrate that even when all of the relevant Weyl modules have at most two composition factors, we do not have to go so far to see some difficult module structures appear.
With this example in mind, it is unclear how much further one can hope to push the results of \cref{thm:Schuralmostss,thm:pdiv1}.

\begin{Example}\label{eg:5factorsummand}
Let $p=3$, $k=7$, and $j=3$.
Using \cref{thm:decompnos,thm:Henkeformula}, we deduce that $\Delta(1^7) \otimes \Delta(1^3)$ (and therefore $\spe{((ke),(je))}$) has two summands, composed from the simple Weyl modules $\Delta(1^{10}) = \sL{1^{10}}$ and $\Delta(2,1^8) = \sL{2,1^8}$, as well as the Weyl modules
\[
\Delta(2^2,1^6) \cong \sL{1^{10}} \hspace{2pt} |\hspace{1pt} \sL{2^2,1^6}
\quad \text{and} \quad
\Delta(2^3,1^4) \cong \sL{2^2,1^6} \hspace{2pt} |\hspace{1pt} \sL{2^3,1^4}.
\]
By checking residues, one can see that $\Delta(2,1^8)$ lies in one block, while the other three Weyl modules lie in another one.
Thus one summand of $\Delta(1^7) \otimes \Delta(1^3)$ is simple, and isomorphic to $\sL{2,1^8}$, while the other one comprises of the remaining three Weyl modules, or five simple modules, and must be self-dual.
It follows that this summand has one of the following two structures.
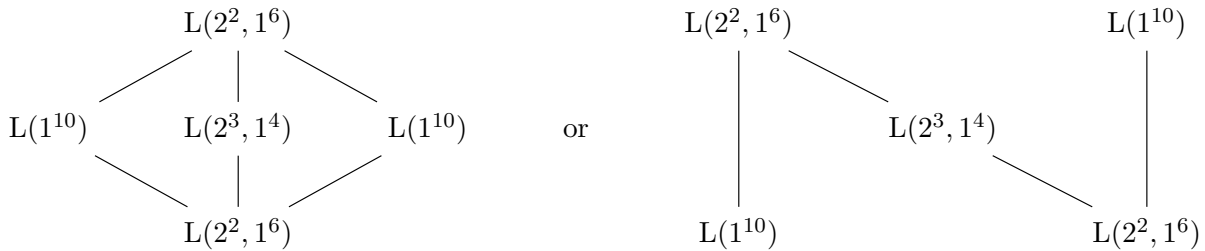
\[
\begin{tikzcd}
	& {\operatorname{L}(2^2,1^6)} &&& {\operatorname{L}(2^2,1^6)} && {\operatorname{L}(1^{{10}})} \\
	{\operatorname{L}(1^{10})} & {\operatorname{L}(2^3,1^4)} & {\operatorname{L}(1^{10})} & {\text{or}} && {\operatorname{L}(2^3,1^4)} \\
	& {\operatorname{L}(2^2,1^6)} &&& {\operatorname{L}(1^{10})} && {\operatorname{L}(2^2,1^6)}
	\arrow[no head, from=3-5, to=1-5]
	\arrow[no head, from=1-7, to=3-7]
	\arrow[no head, from=1-5, to=2-6]
	\arrow[no head, from=2-6, to=3-7]
	\arrow[no head, from=1-2, to=2-2]
	\arrow[no head, from=1-2, to=2-1]
	\arrow[no head, from=1-2, to=2-3]
	\arrow[no head, from=2-1, to=3-2]
	\arrow[no head, from=2-2, to=3-2]
	\arrow[no head, from=2-3, to=3-2]
\end{tikzcd}
\]

But, as in the proof of \cref{prop:j=2}(vi), by \cite[Lemma~A3.1]{Donkin}, $\Delta(1^7) \otimes \Delta(1^3)$ must have $\Delta(1^{10}) = \sL{1^{10}}$ appearing as a quotient module, which rules out the first configuration above.
\end{Example}

%

\bibliographystyle{amsalpha} 
\phantomsection
\addcontentsline{toc}{section}{\refname}
\bibliography{master}

\end{document}